\documentclass[journal]{IEEEtran}

\usepackage{cite}
\usepackage{textcomp}
\usepackage{amsmath}
\usepackage{mathrsfs}
\usepackage{amssymb}
\usepackage{amsthm}
\usepackage{graphicx}
\usepackage{enumitem}
\usepackage{subfigure}
\usepackage{color}
\usepackage{multirow}
\usepackage{verbatim}
\usepackage{algorithm,algorithmicx,algpseudocode}
\usepackage{cite}
\usepackage{bbm}
\usepackage{bigstrut}
\usepackage[left=1in,top=1in,right=1in,bottom=1in,letterpaper]{geometry}
\usepackage[linktocpage,colorlinks,linkcolor=blue,anchorcolor=blue,citecolor=blue,urlcolor=blue]{hyperref}
\usepackage{cleveref}

\numberwithin{equation}{section}

\newtheorem{theorem}{Theorem}[section]
\crefname{theorem}{Theorem}{Theorems}

\newtheorem{lemma}[theorem]{Lemma}
\crefname{lemma}{Lemma}{Lemmas}

\newtheorem{definition}[theorem]{Definition}
\crefname{definition}{Definition}{Definitions}

\newtheorem{assumption}{Assumption}
\crefname{assumption}{Assumption}{Assumptions}

\newcommand{\N}{\mathcal{N}}

\newcommand{\cX}{\mathcal{X}}

\newcommand{\R}{\mathbb{R}}
\newcommand{\E}{\mathbb{E}}

\newcommand{\etal}{ et al. }

\newcommand{\half}{\frac{1}{2}}

\newcommand{\be}{\begin{equation}}
\newcommand{\ee}{\end{equation}}
\newcommand{\ba}{\begin{array}}
\newcommand{\ea}{\end{array}}
\newcommand{\bad}{\begin{aligned}}
\newcommand{\ead}{\end{aligned}}

\newcommand{\normtwo}[1]{\| #1 \|}
\newcommand{\abs}[1]{| #1 |}

\newcommand{\inp}[2]{\left\langle #1, #2 \right\rangle}
\newcommand{\argmin}{\mathop{\rm argmin}}

\newcommand{\dist}{\mathrm{dist}}

\newcommand{\st}{\mathrm{s.t. }}

\newcommand{\prox}{\mathrm{prox}}
\newcommand{\Proj}{\mathrm{Proj}}

\def\BibTeX{{\rm B\kern-.05em{\sc i\kern-.025em b}\kern-.08em
    T\kern-.1667em\lower.7ex\hbox{E}\kern-.125emX}}

\title{ \LARGE On Distributed Non-convex Optimization: Projected Subgradient Method For Weakly Convex Problems in Networks}

\author{Shixiang~Chen,~
	Alfredo~Garcia,~
	and~Shahin~Shahrampour
	\thanks{The authors are with the Wm Michael Barnes '64 Department of Industrial and Systems Engineering, Texas A\&M University, College Station, TX 77843. 
	Email addresses: {\tt\small sxchen@tamu.edu} (S. Chen), {\tt\small alfredo.garcia@tamu.edu } (A. Garcia), {\tt\small shahin@tamu.edu} (S.  Shahrampour).\\ This work was supported by NSF ECCS-1933878.}}

\begin{document}

\maketitle

\begin{abstract}
The stochastic subgradient method is a widely-used algorithm for solving large-scale optimization problems arising in machine learning. Often these problems are neither smooth nor convex. Recently, Davis\etal\cite{davis2019stochastic,davis2018subgradient} characterized the convergence of the stochastic subgradient method for the weakly convex case, which encompasses many important applications (e.g., robust phase retrieval, blind deconvolution, biconvex compressive sensing, and dictionary learning). In practice, distributed implementations of the projected stochastic subgradient method (stoDPSM) are used to speed-up risk minimization. In this paper, we propose a distributed implementation of the stochastic subgradient method with a theoretical guarantee. Specifically, we show the global convergence of stoDPSM  using the Moreau envelope stationarity measure. Furthermore, under a so-called sharpness condition, we show that deterministic DPSM (with a proper initialization) converges linearly to the sharp minima, using geometrically diminishing step-size. We provide numerical experiments to support our theoretical analysis. 
\end{abstract}


\IEEEpeerreviewmaketitle

\section{Introduction}
\label{sec:introduction}
Optimization in multi-agent networks has received a great deal of attention in the past few years in control, signal processing, and machine learning. A wide range of networked problems such as distributed detection \cite{shahrampour2016distributed}, estimation \cite{kar2012distributed}, and localization \cite{atanasov2014joint} can be formulated via distributed optimization with applications in wireless sensor networks \cite{mateos2012distributed}, robotic networks \cite{bullo2009distributed}, power networks \cite{bolognani2014distributed}, and social networks \cite{acemoglu2008convergence}. In such decentralized frameworks, a number of agents in a network need to accomplish a global task, which is formulated as an optimization. Each individual agent, however, has limited information about the objective function. Therefore, agents locally interact with each other to solve the global problem. Decentralized techniques have gained popularity over time due to robustness to individual failures, imposing low computational burden on individual agents, and promoting privacy. 

In a (constrained) multi-agent optimization, we deal with a problem of the form
\be \label{opt:pro_wcvx}
\min_x f(x) =\frac{1}{N} \sum_{i=1}^N f_i(x)\quad \st \quad x\in \cX,
\ee
where  $\cX\subset\R^n$ is a closed convex set known to all agents, and $f_i(x)$ is only available to agent $i$. Given this partial knowledge, agents must communicate with each other to minimize $f(x)$. There exists a large body of works on distributed convex optimization, where each $f_i(x)$ is convex (see e.g., the seminal work of \cite{nedic2009distributed} and its following papers). The literature has witnessed various algorithms for solving \eqref{opt:pro_wcvx}, which come with theoretical guarantees. In this paper, we depart from the classical convex setting and focus on {\it weakly convex and non-smooth} problems. In particular,  we  assume that every $f_i(x)$ is $\rho-$weakly convex 
 and the subgradient $\|\partial f_i(x)\|$ is uniformly bounded. {The definition of $\rho-$weakly convex function is as follows.}
\begin{definition}
  A function $f(x)$ is $\rho-$weakly convex ($\rho > 0$) if	there exists a convex function $h(x)$ such that $f(x)  = h(x) - \frac{\rho}{2}\|x\|^2$.
\end{definition}
Weakly convex problems play a key role in important machine learning applications such as robust phase retrieval \cite{davis2017nonsmooth,duchi2019solving,eldar2014phase}, blind deconvolution \cite{chan1998total,levin2011understanding}, biconvex compressive sensing \cite{ling2015self}, and dictionary learning \cite{davis2019stochastic}. Recently, Davis\etal\cite{davis2019stochastic,davis2018subgradient} characterized the convergence of the stochastic subgradient method for the weakly convex case. However, training the aforementioned models on a single device can take a significant amount of time for large-scale data. In practice, distributed implementations of the stochastic (sub)gradient method (stoDPSM) are used to speed-up the training time (see e.g. \cite{Tensorflow,Poseidon}) and they can achieve a linear speedup \cite{lian2017can}. 

In this paper, we focus on developing a theoretical  convergence result for the distributed projected subgradient method (DPSM)
\be \label{alg:dis_subgrad_orgin}
x_{i,k+1} = \Proj_{\mathcal{X}}\left(\sum_{j=1}^N a_{i,j}(k) x_{j,k}  - \alpha_k g_{i,k}\right),
\ee
where $\Proj_{\mathcal{X}}$ denotes the orthogonal projection onto $\cX$,  $a_{i,j}(k)$ is the weight agent $i$ associates to information received from agent $j$ at time $k$, $\alpha_k$ is the stepsize, and $g_{i,k}$ is any subgradient of $f_i$ evaluated at $\sum_{j=1}^N a_{i,j}(k) x_{j,k}$. This algorithmic  scheme was first proposed for unconstrained distributed convex optimization \cite{nedic2009distributed} and it was then extended to the constrained scenario \cite{nedic2010constrained,liu2017convergence}. Under standard assumptions on the weights $a_{i,j}(k)$ and $\alpha_k$, it can be shown that \eqref{alg:dis_subgrad_orgin} converges to a minimizer of $f(x)$ in \eqref{opt:pro_wcvx} for convex problems. Nevertheless, the weakly convex problem (which is essentially non-convex) has not been addressed in the literature.  (i)  We provide global convergence results for DPSM using the notion of Moreau envelope and show that the infimum of its gradient approaches zero with a rate $\mathcal{O}(k^{-1/4})$ {if the stepsize is $\alpha_k=\mathcal{O}(1/\sqrt{k})$}. (ii) We show a linear convergence rate of DPSM under a so-called sharpness condition. Specifically, the linear rate also relies on the fact that each local variable is sufficiently close to a sharp minimizer and  the  stepsize is given by $ \alpha_k=\mu_0 \gamma^k, \mu_0\in(0,1), \gamma\in(0,1)$, where  $\mu_0$ and $\gamma$ also depend on both function and network parameters. (iii) We also extend the global convergence results to the stochastic   setting (stoDPSM). This paper is the first work providing convergence analysis of {\it distributed} subgradient for {\it weakly convex, non-smooth} problems, extending the classical convex analysis \cite{nedic2009distributed} to a {\it non-convex} setting. In other words, it generalizes the centralized results of \cite{davis2019stochastic,davis2018subgradient} to the decentralized setting. The technical challenges are as follows. For the global convergence of (sto)DPSM using Moreau envelope, we show a novel and crucial weakly convex inequality in \cref{lem:alt_weak_cvx}. For proving the linear convergence rate under the presence of sharpness property, we need to carefully identify the relationship between the geometrically diminishing stepsize with the function and network parameters. Moreover, since the sharpness property holds in a local region,  each local variable should stay in the neighborhood of that region, posing another technical challenge.

\subsection{Related work}
We now briefly review the existing work on distributed (sub)gradient method. When $f_i$ is  convex and non-smooth,  the distributed  subgradient method converges in terms of function value\cite{nedic2009distributed} in unconstrained  scenario
and the distributed stochastic subgradient projection algorithms\cite{ram2010distributed} can deal with a common constraint. In the case that each agent only knows its own constraint information, convergence guarantee was  proved in \cite{nedic2010constrained}. In all above results, a diminishing stepsize (that is square-summable but not summable) is required and the  convergence for constrained problem is  measured by the distance between the  sequence and the optimal set, i.e., the limit of  $\dist(x_{i,k},\cX^*)$ is zero for any $i$, where $\cX^*$ is the optimal set.  The square-summable condition was  relaxed in \cite{liu2017convergence}, provided the optimum set is bounded. {The best convergence rate of non-smooth convex problem is given by $\inf_{0\leq k\leq T} f(\bar x_k) - f^* = \mathcal{O}(\log T/\sqrt{T}) $ with $\alpha_k = \mathcal{O}(1/\sqrt{k})$\cite{nedic2009distributed}, where $\bar x_k :=  1/N\sum_{i=1}^{N}x_{i,k}$ is the average point and $f^*$ is the optimal function value.} Moreover, a  convergence rate  of $\dist(x_{i,k},\cX^*) = \mathcal{O}(1/\sqrt{k})$ is shown in \cite{liu2017convergence} if the stepsize is set to $\alpha_k=\mathcal{O}(1/k)$ for strongly convex $f_i$.  
For smooth convex and unconstrained case, the convergence is  established in\cite{yuan2016convergence}.  
We refer to the survey \cite{nedic2018network} for a complete review of decentralized optimization of convex problems.

If $f_i$ is non-convex  and its gradient  $\nabla f_i$ is Lipschitz continuous,  $f_i$ is automatically weakly convex. Algorithms such as \cite{bianchi2012convergence,lian2017can,zeng2018nonconvex,assran2019stochastic,scutari2019distributed} have been proposed for the non-convex setting. For example,  the convergence of  distributed projected stochastic gradient was discussed in\cite{bianchi2012convergence}; an ergodic convergence rate was established in \cite{zeng2018nonconvex} for proximal  gradient method. A push-sum stochastic gradient method was proposed to train deep neural networks in \cite{assran2019stochastic} and the sublinear rate was established. 
When the objective is non-smooth, previous work, e.g., \cite{zeng2018nonconvex,scutari2019distributed} studied the composite form, i.e., $f_i(x)= g_i(x) + h_i(x),$ where $g_i$ is smooth but $h_i$ is non-smooth with a closed-form proximal mapping.  In contrast, the non-smooth objective in \eqref{opt:pro_wcvx}  generally does not follow an easy proximal mapping. When the computation of subgradient of $f_i$  is inexpensive, algorithm \eqref{alg:dis_subgrad_orgin} is a better choice. 

Recently, for weakly convex $f_i$, the  centralized proximal-type subgradient methods have been shown to converge in finite time in terms of a stationarity measure using Moreau envelope (see \cite{davis2019stochastic,duchi2019solving}). Under the presence of sharpness property, the centralized subgradient converges linearly in a local region \cite{davis2018subgradient}. We summarize the convergence results for distributed projected subgradient method in  \cref{table:convergence}.

\begin{table*}[t]
	\centering
	\begin{tabular}{|c|c|c|c|}
		\hline
		$f_i$        &  strongly convex                    &  convex              								& weakly convex        \\ \hline
		Measure & $\dist(x_{i,k},\cX^*)$	&   $\dist(x_{i,k}, \cX^*)$ or  $\inf\limits_{1\leq k\leq T} f(\bar x_k)- f^*$				&    $\|\nabla \varphi_t(\bar{x}_k)\|$	             \\    \hline
	          	&                &          &             		                                                      	This paper:	    \\ 
	 Convergence      & $\dist(x_{i,k},\cX^*)=\mathcal{O}(\frac{1}{\sqrt{k}})$  	&       $\lim\limits_{k\rightarrow\infty}  \dist(x_{i,k}, \cX^*) =0$          \cite{nedic2010constrained,liu2017convergence};             
	                                 &   $\inf\limits_{1\leq k\leq T}\|\nabla\varphi_t(\bar{x}_k) \|=\mathcal{O}(\frac{1}{{T}^{1/4}}) $	\\
		&    with $\alpha_k=\mathcal{O}(1/k)$  \cite{liu2017convergence}        &  {  $\inf\limits_{1\leq k\leq T} f(\bar x_k)- f^* = \mathcal{O}(\frac{\log T}{\sqrt{T}}) $ with  $\alpha_k = \mathcal{O}(\frac{1}{\sqrt{k}})$\cite{nedic2009distributed} }
		&       with $\alpha_k = \mathcal{O}(1/\sqrt{k})$.	   \\ \hline
	\end{tabular}
	\caption{Convergence results: distributed projected subgradient method for constrained optimization}\label{table:convergence}
\end{table*}

\section{Preliminaries}

{\bf Notation:}
We use $\inp{x}{y}=x^\top y$ to denote the Euclidean inner product and $\|x\|$ to denote the Euclidean norm of $x$. We denote by $\partial h(x)$ the subgradient set of a convex function $h(x)$. Abusing notation, we use $\inp{\partial h(x)}{y}$ to denote the inner product of any elements of $\partial h(x)$ and a vector $y$. 

\subsection{Network Model}
We consider a time-varying network of agents that can exchange information locally. To model the network, we use a time-varying graph $(\mathcal{V},E_k)$, where $\mathcal{V}=\{1,\dots ,N\}$ denotes
the set of nodes and $E_k\subseteq \mathcal{N}%
\times \mathcal{N}$ is the set of links connecting nodes at time $k>0$. 
Let $A(k)=[a_{i,j}(k)]$ denote the matrix of weights associated with links in the graph at time $k>0$.  For node $i$, $\N_i(k)$ denotes the neighborhood of $i$ in which $a_{i,j}(k)>0$.
Define $\Phi(k,s) = A(s)A(s+1)\cdots A(k-1)A(k)$ for $k\geq s$, $\Phi(k,k)=A(k)$ and $\Phi(k,s)=I$ for $k<s$.  
\subsection{Weak Convexity, Optimality Measure and Sharpness}
We assume the local objective $f_i(x)$ in \eqref{opt:pro_wcvx} is $\rho-$weakly convex for some $\rho\geq 0$; i.e., there exists a convex function $h_i(x)$ such that $f_i(x)  = h_i(x) - \frac{\rho}{2}\|x\|^2$. Although $f_i(x)$ is not convex, we may define its subdifferential by 
\be \label{def:subgrad}
\partial f_i(x)  = \partial h_i(x) - \rho x, \quad  \forall x \in \cX;
\ee
(see \cite{vial1983strong}). Here, $\partial h_i(x)$ is the subdifferential in the convex sense. 
The following lemma states an equivalent definition of weakly convex functions and strongly convex functions. {We are particularly interested in \eqref{ineq:alt_def_weak_cvx} for the analysis and we establish \eqref{ineq:alt_def_strongly_cvx} to prove \eqref{ineq:alt_def_weak_cvx}.} The proof is given  in the Appendix. 
\begin{lemma}\label{lem:alt_weak_cvx}
	If $f(x)$ is $\rho$-weakly convex and $g(x)$ is $\tau-$strongly convex in $\R^n$, then $\forall x_1,\ldots, x_m\in\R^n$, it follows that
	\be \label{ineq:alt_def_weak_cvx}
	f(\sum_{i=1}^m a_i x_i) \leq \sum_{i=1}^m a_i f(x_i) + \frac{\rho}{2}\sum_{i=1}^{m-1} \sum_{j= i+1}^m a_i a_j \|x_i-x_j\|^2
	\ee
	and 
	\be \label{ineq:alt_def_strongly_cvx}
	g(\sum_{i=1}^m a_i x_i) \leq \sum_{i=1}^m a_i g(x_i) -\frac{\tau}{2}\sum_{i=1}^{m-1} \sum_{j= i+1}^m a_i a_j \|x_i-x_j\|^2,
	\ee
	where $\sum_{i=1}^m a_i=1$ and $a_i\geq 0$ for all $i$.
\end{lemma}
 
To analyze DPSM, we follow the framework of \cite{davis2019stochastic}, where a novel convergence analysis for {\it centralized} subgradient method is proposed. We extend this analysis to the distributed case for weakly convex problems. Since there  exist  different stationary points in non-convex problems, neither  the suboptimal objective  value  nor the  distance to the optimum set tend to be good measures for the analysis. On the other hand, the subgradient of the objective is not continuous, which makes it difficult to analyze the convergence of the subgradient norm.  A surrogate stationary measure for problem \eqref{opt:pro_wcvx} was thus defined using the Moreau envelope in \cite{davis2019stochastic}. We briefly review it in the sequel.

  Recall that   $f_i(x)$ is $\rho-$weakly convex, iff we  have the following inequality \cite[Lemma 2.1]{davis2019stochastic}
\be\label{ineq:subg_weak_cvx} f_i(y)\geq f_i(x) + \inp{ \partial f_i(x)}{y-x} - \frac{\rho}{2} \normtwo{y-x}^2. \ee
This inequality is also known as prox-regular inequality  introduced earlier in \cite{poliquin1996prox}.
Therefore, $f(x) =1/N \sum_{i=1}^{N}f_i(x)$ is also $\rho-$weakly convex. Let $\varphi(x) = f(x)+ \mathbb{I}_{\cX}(x)$, where $\mathbb{I}_{\cX}$ is the indicator function of $\cX$ \footnote{$\mathbb{I}_{\cX}(x)=0$ when $x\in \cX$, and $\mathbb{I}_{\cX}(x)=\infty$ otherwise.}. 
Define the Moreau envelope \cite{Rockafellar-book-70} as
\[
\varphi_{t}(x):=\min_{y\in\R^n} \varphi(y) + \frac{1}{2 t}\|y-x\|^{2}, \quad  0<t< 1/\rho.
\]
Since $t<1/\rho$ and $f_i$ is $\rho-$weakly convex, we have that $\varphi(y) + \frac{1}{2 t}\|y-x\|^{2}$ is strictly convex, i.e., the minimization problem above is strictly convex and the minimizer is unique. The Moreau envelope is a $C^1$ smooth approximation to the non-smooth function $f(x)$ over $\cX$. 
Let us define
\[
\hat{x} =\argmin_{y\in\R^n} \varphi(y)+\frac{1}{2 t}\|y-x\|^{2},
\]
where  $\prox_{tf}(x) :=\hat{x}$ is called the proximal mapping. Here, we omit $\cX$ in the notation $\prox_{tf}$, but recall that the proximal mapping is related to $\cX.$ The proximal mapping is only used in the analysis, and it is not computed in the algorithm.
From the optimality condition of $\hat{x}$, one has 
\[ 0\in \partial f(\hat{x}) + \partial \mathbb{I}_{\cX}(\hat{x})+ \frac{1}{t} (\hat{x}-x) .\]

It follows  that
\[ \dist(0,\partial f (\hat{x}) + \partial \mathbb{I}_{\cX}(\hat{x}) ) \leq    \frac{1}{t} \normtwo{\hat{x}-x}.\]
Therefore, if $\frac{1}{t} \normtwo{\hat{x}-x}\leq \varepsilon$, then $\hat{x}$ is $\varepsilon-$stationary and $x$ is close to the $\varepsilon-$stationary point $\hat{x}$. We also have  \cite{Rockafellar-book-70}
\be\label{def:surrogate measure}
\|\nabla \varphi_t({x})\|= \frac{1}{t} \normtwo{\hat{x}-x},
\ee   which can be used as near-stationarity  measure of $x$. 
{
Next, we introduce the  sharpness property.
\begin{definition}[Sharpness]\label{def:sharpness}
  A function $f:\cX \rightarrow \R $ possesses the  local sharpness property, if  there exist  constants $\beta>0$ and $B>0$ such that the following inequality holds  for a minimizer  $x^*$ of $f(x)$
  \begin{equation}\label{def:sharp}
   f(x)- \min f\geq \beta \|x-x^*\|, \quad  \forall x\in \mathcal{B},
  \end{equation}
  where $\mathcal{B} = \{x\in\cX: \|x-x^*\|\leq B \}.$
  Furthermore, $x^*$ is called the local sharp minimizer of $f$.
\end{definition}
As we can see, the sharpness property intuitively suggests that the objective grows at least linearly in $\mathcal{B}$. It has been shown that the centralized subgradient method converges linearly in the neighborhood of the local sharp minimizer\cite{davis2017nonsmooth,davis2018subgradient,li2019nonsmooth}. 
}
\subsection{Assumptions}
 In this part, we introduce the assumptions used for our analysis. To begin with, we define the average vector
 \[
 \bar{x}_{k}:=\frac{1}{N} \sum_{i=1}^{N} x_{i, k},
 \]
 and
 \[ v_{i,k} := \sum_{j\in\N_i(k)} a_{i,j}(k) x_{j,k}. \]
 Then, the iteration in DPSM \eqref{alg:dis_subgrad_orgin} can be rewritten as
 \be\label{alg:proj_sub}  x_{i,k+1} = \Proj_{\mathcal{X}}\left(  v_{i,k} - \alpha_k g_{i,k}\right),   \ee
 where  $g_{i,k} \in \partial f_i(v_{i,k})$ is any element of the subdifferential set. Unlike the centralized algorithm,  the distributed update does not rely on a fusion center and computing $v_{i,k}$ and $g_{i,k}$ can be done in a decentralized manner. 
The following assumptions on the network are commonly adopted in the literature\cite{tsitsiklis1984problems,nedic2009distributed,nedic2010constrained}.
\begin{assumption}[Weights rule]\label{assup:graph1}
	There exists a scalar $\eta\in(0,1)$ such that  for all $i,j\in\{1,\dots,N\}$,
	\begin{itemize}
    \item $a_{i,i}(k)\geq \eta$ for all $k\geq 0$.
    \item If $a_{i,j}(k)>0$ then $a_{i,j}(k)\geq \eta$.
	\end{itemize}
\end{assumption}

\begin{assumption}[Doubly stochasticity]\label{assup:graph2}
	The weight matrix $A(k)$ is doubly stochastic (i.e., $\sum_j a_{i,j}(k)=\sum_j a_{j,i}(k)=1, \forall i,k$).
\end{assumption}
\begin{assumption}[Connectivity]\label{assup:graph4}
	The graph $(\mathcal{V},E_\infty)$ is strongly connected, where $E_\infty$ is the set of edges $(j,i)$ representing agent pairs communicating directly infinitely many times, i.e.,
	$E_\infty = \{(j,i): (j,i)\in E_k \ \textit{for infinitely many indices k}\}$.
	\end{assumption}
\begin{assumption}[Bounded intercommunication interval]\label{assup:graph3}
There exists an integer $B\geq1$ 
such that for every $(j, i) \in E^\infty$, agent $j$ sends its information to the neighboring
agent $i$ at least once every $B$ consecutive time slots, i.e., at time $t_k$ or at time $t_k+1$ or
$\ldots$ or (at latest) at time $t_k+B-1$ for any $k \geq 0$.
\end{assumption}

We also need some assumptions on the function $f(x)$.
\begin{assumption}\label{assup:bound_subgrad} 
    	$f(x)$ is lower bounded. 
	Every $f_i$ is $\rho-$weakly convex and we have  $\|\partial f_i(x)\|\leq L$  for $x\in \cX$. 
\end{assumption} 
The bounded subgradient holds if $f_i(x)$ is globally $L-$Lipschitz continuous on $\R^n$ or $\mathcal{\mathcal{X}}$\cite{Rockafellar2009}. 
One common class of weakly convex functions is $f(x)=h(c(x))$, where $h$  is convex and Lipschitz and $c$ is smooth with a Lipschitz continuous Jacobian \cite{drusvyatskiy2019efficiency}. However, such $f(x)=h(c(x))$ is usually locally Lipschitz continuous. A common assumption to resolve this issue is that $\mathcal{X}$ is compact or the sequence $\{x_{i,k}\}$ is bounded. Then, the boundedness of subgradient is equivalent to the $L-$Lipschitz property for $f_i(x)$.  Such assumptions are usually needed in centralized algorithms (see \cite{davis2019stochastic}).

Last, the stepsize for the global convergence of DPSM should be non-summable  and diminishing.
\begin{assumption}\label{assup:stepsize}
	The stepsize $\alpha_k>0$ satisfies
	\[ \sum_{k=0}^\infty \alpha_k = \infty,\     \lim_{k\rightarrow\infty}\alpha_k = 0 \  \textit{and}\   {\lim_{k\rightarrow \infty}\frac{\alpha_{k+1}}{\alpha_k}=1}. \] 
\end{assumption}
A commonly used stepsize sequence satisfying \cref{assup:stepsize} can be as follows	\[ \alpha_k = \frac{1}{k^q},\quad  \text{where}\  q\in(0,1]. \] 
\subsection{Technical lemmas}
In this part, we introduce some necessary  lemmas. All proofs can be found in the Appendix. 
\begin{lemma}\label{lem:linear_rate_of_Phi} \cite[Proposition 1]{nedic2009distributed}
		Under  \cref{assup:graph1,assup:graph2,assup:graph3,assup:graph4}, 
 there exist constants $c>0$ and $\lambda\in(0,1)$ such that
 \[    \|\Phi(k,s) - \frac{1}{N}\mathbf{11}^\top\|_{\text{op}}  \leq  c\lambda^{k-s},\]
 where $\|\cdot\|_{op}$ is the matrix operator norm.
\end{lemma}

\begin{lemma}\label{lem:finite_sum}\cite[Lemma 7]{nedic2010constrained}\cite[Proposition 8]{liu2017convergence}
	Let $\lambda\in(0,1)$ and $\{\gamma_k\}$ be a positive sequence.  Suppose $\gamma_k$ satisfies \cref{assup:stepsize}. Considering the convolution sequence $\sum_{k=0}^{T-1}  \lambda^k\gamma_{T-k-1}$, we have
	\be\label{ineq:estimation_convolution}
	\sum_{k=0}^{T-1}  \lambda^k\gamma_{T-k-1} =  {\mathcal{O}(\frac{\gamma_{T-1}}{1-\lambda})}.
	\ee
\end{lemma}

For convergence analysis, we should show that the deviation of individual errors from the mean  $\|x_{i,k} - \bar{x}_k\|$ goes to zero.  We define a vector $\Delta_k$ where $\Delta_{k,i} :=  x_{i,k}-\bar{x}_k$. That is,  $\Delta_k\in\R^{Nn}$ is the vector formed by stacking all individual deviations from the mean. 
The following inequality \eqref{lem:consensun_ineq1} 
for $\Delta_k$ was established in  \cite{nedic2010constrained,liu2017convergence,nedic2018network}  for convex problems. We show that it still holds for the weakly convex case.
\begin{lemma}\label{lemma:consensus1} 
	Under  \cref{assup:graph1,assup:graph2,assup:graph3,assup:graph4,assup:bound_subgrad,assup:stepsize}, 
	for the distributed projected subgradient algorithm \eqref{alg:proj_sub}, we have the following consensus result
	\be \label{lem:consensun_ineq1}
	\lim\limits_{k\rightarrow \infty} \normtwo{\Delta_{k,i} } = 0, \quad \forall i.
	\ee
\end{lemma}

Furthermore, similar to the result of \cite[Proposition 8]{liu2017convergence}, the convergence rate can be characterized as follows. 
\begin{lemma}\label{lem:consensus}
	Under  \cref{assup:graph1,assup:graph2,assup:graph3,assup:graph4,assup:bound_subgrad,assup:stepsize}, 
for the distributed projected subgradient algorithm \eqref{alg:proj_sub}, we have the following error rate
\be\label{lem:consensus_order}
\|\Delta_k\|  = {\mathcal{O}(\sqrt{N}L\cdot\frac{\alpha_k}{1-\lambda})}.
\ee
\end{lemma}

We also have the following well-known property of the projection onto convex sets.
\begin{lemma}\cite{nedic2010constrained}\label{lem:property_projection}
	For the convex closed set $\cX$, it follows that $\forall y\in \cX$
	\[ \|\Proj_{\mathcal{X}}(x)-y\|^2\leq  \|x-y\|^2 - \|x-\Proj_{\mathcal{X}}(x)\|^2. \]
\end{lemma}

We further have  the  following property of the proximal mapping. 
Although the proximal mapping   is not non-expansive when $f(x)$ is   weakly convex,  it is still Lipschitz continuous. This can be easily proved using the same idea for convex functions \cite{moreau1965proximite}. We omit the proof.  
\begin{lemma}\label{lem:nonexpansive:proximal_map}
	If $f(x)$ is $\rho-$weakly convex, then the proximal mapping  with $t<1/\rho$ satisfies
	\[\bad
	 \| \prox_{tf}(x_1) - \prox_{tf}(x_2) \|\leq \frac{1}{1-t\rho}&\|x_1-x_2\|,\\
	 &\quad \forall  x_1,x_2\in\cX.
	 \ead\]
\end{lemma}

\section{Main Results and Convergence Analysis}
{In this section, we state the main convergence results of DPSM. First, if the  \cref{assup:graph1,assup:graph2,assup:graph4,assup:graph3,assup:bound_subgrad,assup:stepsize} hold, the   Moreau envelope sequence $\{\varphi_t(\bar{x}_k)\}$ converges and the infimum of its  gradient converges to $0$. Second, under the sharpness condition, DPSM converges linearly with a geometrically diminishing stepsize in a neighborhood of a sharp minimizer.  Finally, we also provide the convergence result of distributed projected stochastic subgradient method.}
\subsection{Global Convergence}
We now establish the first convergence result. The following lemma states the improvement after one iteration of the algorithm \eqref{alg:proj_sub}. The proof is given  in the Appendix. 
\begin{lemma}[One-step improvement]\label{lem:one-step}
	Let 
	\begin{align*}
	s_k&:=\argmin_{y\in\cX} f(y) + \frac{1}{2t}\|y-\bar{x}_k\|^2,\\
	\hat{v}_{i,k} &:=\argmin _{y\in\cX} f(y)+\frac{1}{2 t}\|y-v_{i,k}\|^{2}.
	\end{align*}
	Under  \cref{assup:bound_subgrad}, we have
	\be\label{ineq:one_step}
	\bad
	&\quad\sum_{i=1}^N \left\|x_{i , k+1}- 	\hat{v}_{i,k} \right\|^{2} \leq   \sum_{i=1}^N \left\|{v}_{i , k}- 	\hat{v}_{i,k} \right\|^{2} \\
	&\quad+ 2\alpha_k\left(   N(-\frac{1}{2t} +\rho ) \normtwo{\bar{x}_k -s_k  }^2 \right.\\
	&\quad \left.+ \frac{L(2-t\rho)}{1-t\rho} \sum_{i=1}^N \normtwo{x_{i,k} - \bar{x}_k  } \right.\\
	&\quad \left.+   2\rho(1+ \frac{1}{(1-t\rho)^2})\sum_{i=1}^N\normtwo{ x_{i,k}  -\bar{x}_k  }^2 \right) + N L^2\alpha_k^2.
	\ead 
	\ee
\end{lemma}

By invoking Lemma \ref{lemma:consensus1}, the distance $\|x_{i,k}-\bar{x}_k\|$  goes to $0$ for all $i$ when  $\alpha_k$   goes to zero.
Note that for  $0<t<1/(2\rho)$, the term $(-\frac{1}{2t}+\rho)\|\bar{x}_k-s_k\|^2$ in \eqref{ineq:one_step} is strictly negative  if  $\|\bar{x}_k-s_k\|^2$ is not zero.  Then,  we may have $\sum_{i=1}^N \left\|x_{i , k+1}- 	\hat{v}_{i,k} \right\|^{2} <   \sum_{i=1}^N \left\|{v}_{i , k}- 	\hat{v}_{i,k} \right\|^{2}$, which means that comapred to $v_{i,k}$, the new point $x_{i,k+1}$ is closer to $\hat{v}_{i,k}$. Hence, the algorithm continues to make progress.
Now, we present our first main result, which shows the decay of the gradient of the Moreau envelope (optimality measure). The proof is given  in the Appendix. 
\begin{theorem}\label{thm:convergence_1}
	Let $0<t<\frac{1}{2\rho}$ and  $\{ x_{i,k}\}$ be the sequence of projected subgradient method for solving problem \eqref{opt:pro_wcvx}.  
	Under  \cref{assup:graph1,assup:graph2,assup:graph4,assup:graph3,assup:bound_subgrad,assup:stepsize},  
	\begin{enumerate}
		\item[(1)] If $\sum_{k=0}^\infty \alpha_k^2 <\infty$, there exists $\bar{\varphi}_t$ such that  $\lim\limits _{k\rightarrow \infty} \varphi_{t}(x_{i,k})=\lim\limits _{k\rightarrow \infty} \varphi_{t}(\bar{x}_k) =\bar{\varphi}_t;$ 
		\item[(2)]  There exists  {$  b_k=\mathcal{O}(\frac{L^2\alpha_k^2}{(1-\lambda)^2})$ } such that 
		\small{
		\begin{align*}
&\inf_k \Vert \nabla \varphi_t(\bar{x}_k)\Vert^2 \\
\leq &\frac{2}{1-2t\rho} \cdot \frac{ {\varphi}_{t}(\bar{x}_0)- \inf {\varphi}_{t}(x)+\sum\limits_{k=0}^{\infty} b_{k}+\sum\limits_{k=0}^{\infty}\frac{L^{2}\alpha_{k}^{2}}{2t}
}{ \sum_{k=0}^{\infty} \alpha_k}.
		\end{align*}
   }
	\end{enumerate}

\end{theorem}
Statement (1) of \cref{thm:convergence_1} suggests that the Moreau envelope function value  converges  at the mean $\bar{x}_k$ if $\sum_{k=0}^\infty \alpha_k^2 <\infty$. Statement (2) is the decentralized counterpart of the centralized algorithm established in \cite{davis2019stochastic}. 
It also provides the convergence rate of $\inf_k \Vert \nabla \varphi_t(\bar{x}_k)\Vert^2$. 
 For example, if $\alpha_k = \mathcal{O}(1/\sqrt{k})$, we have \[ \inf_{1\leq k\leq T} \Vert \nabla \varphi_t(\bar{x}_k)\Vert^2 = { \mathcal{O}(\frac{\bar{\varphi}_{t,0}-  \bar{\varphi}_{t}}{\sqrt{T}} + \frac{L^2}{(1-\lambda)^2}\cdot\frac{\log T}{\sqrt{T}}) } \]
for sufficiently large $T$. Compared with the centralized algorithm\cite{davis2019stochastic}, we have an extra term $\sum_k b_k$ in the upper bound, which is the cost of decentralization as it has the constant  $\frac{1}{(1-\lambda)^2}$ involving network parameters. For the convex problem, a similar result is established under the optimality measure  $\inf_{1\leq k\leq T} f(\bar x_k) - f^*$ \cite[Theorem 8]{nedic2018network}. But the dependence on the $\lambda$ is  $\frac{1}{1-\lambda}$. Whether $\frac{1}{(1-\lambda)^2}$ is optimal for weakly-convex problems is an interesting question that we leave for future work.

\subsection{Local Convergence Rate with Sharpness Property}
In this section, we discuss the convergence rate of  the Algorithm \ref{alg:dis_subgrad_orgin} under the  presence of  sharpness property defined in \cref{def:sharpness}. 
It has been shown the centralized subgradient method converges linearly in the neighborhood of a sharp minimizer \cite{davis2017nonsmooth,davis2018subgradient}, if the Polyak stepsize \cite{polyak1969minimization} or geometrically diminishing stepsize \cite{goffin1977convergence} are adopted. The Polyak stepsize \cite{polyak1969minimization} and geometrically diminishing stepsize were firstly proposed for convex problems and they also work for weakly convex problems. {Since the Polyak stepsize needs the knowledge of the optimal function value and the full information of the objective $f$, } we will only consider the  geometrically diminishing stepsize, i.e, $\alpha_k = \mu_0 \gamma^k$, where $\mu_0>0$ and $\gamma\in(0,1)$ are constants decided by the problem parameters. Under some conditions, we can show the linear rate of DPSM in \cref{thm:linear_rate}.  The proof idea is as follows.  As the sharpness is a property of the whole function $f(x)$, we can only  use the sharpness inequality at $\bar{x}_k$.   We first need to estimate the deviation from mean $\|\Delta_{k}\|$ when using the geometrically diminishing stepsize.

\begin{lemma}\label{lem:estm_geo_step}
 Let  the stepsize $\alpha_k$ in Algorithm \ref{alg:dis_subgrad_orgin} be  $\alpha_k = \mu_0 \gamma^k, \ k\geq 0$, where $\mu_0>0$, $\gamma\geq \lambda^\delta$, $\delta \in(0,1)$  and  $\lambda$ is the parameter given in \cref{lem:linear_rate_of_Phi}. Then, $\|\Delta_k\| = \mathcal{O}(\alpha_k)$.
\end{lemma}
\begin{proof}
	  From inequality \eqref{ineq:estimation_of_Delta} in the proof of \cref{lemma:consensus1} and the fact $\gamma\geq \lambda^\delta $,  we have
	 \be\label{ineq:estimation_of_Delta_2}
	 \bad
	 &\quad \|\Delta_{k+1}\|\\
	 & \leq   c\lambda^{k}\|\Delta_0\| +  c\sqrt{N}L \sum_{l=0}^{k-1} \lambda^{k-l-1} \alpha_{l} + \sqrt{N}L\alpha_{k}\\
	 &\leq  \left( \frac{c}{\lambda}  \|\Delta_0\|  +  \frac{\sqrt{N}L}{\lambda^2}( c \frac{\gamma^{\frac{1}{\delta}-1}}{1- \gamma^{\frac{1}{\delta}-1}} + \lambda)\mu_0 \right)\gamma^{k+1}.
	 \ead\ee
\end{proof}

\begin{assumption}\label{assup:geo_step} 
    Let $x_{1,0},\ldots,x_{N,0}$ be the initial points in Algorithm \ref{alg:proj_sub}.
    Given any constants $\Lambda\in(\lambda,1)$ and  $\Gamma\geq \sqrt{2}$, define 
    {\small
    \begin{align*}
    e_0&:=\min\left\{\max\left\{\frac{\beta}{\rho\Gamma}, \sqrt{\frac{1}{N} \sum_{i=1}^N\|{x}_{i,0} - x^*\|^2}\right\},\frac{B}{\Gamma}\right\},\\
    	a&:=\frac{2 (L+\beta)L}{\lambda^2},\\
    	q &:=  \frac{2\beta }{\Gamma } e_0- \rho e_0^2-\frac{2(L+\beta) c}{\sqrt{N}\lambda}\|\Delta_0\|, 
    \end{align*} }
    where $c,\lambda$ are constants given in \cref{lem:linear_rate_of_Phi}, $\beta$ and $B$ are defined in \eqref{def:sharp}, $\rho$ is the weak-convexity parameter, and $L$ is the bound on subgradients. Let the stepsize in Algorithm \ref{alg:proj_sub} be given by $\alpha_k = \mu_0 \gamma^k,$ where   $0<\mu_0\leq \min\{\frac{e_0}{2\beta - \rho e_0},\frac{q}{10\sqrt{N}(a\lambda +L^2+\frac{ ac\Lambda}{1-\Lambda})}\}$ and $\gamma\in(0,1).$
\end{assumption}
	We use the stepsize assumption above to prove the following theorem. The proof sketch is as follows. The sharpness property holds for the global objective $f(x)$. This motivates us to consider the full information at average point $\bar x_k$. With the help of  Lemma III.3 and Lemma A.1 in the Appendix, we can show that $\sum_{i=1}^N\|{x}_{i,k} - x^*\|^2$ decays linearly, but not for     $\| x_{i,k} - x^*\|$,  $i\in[N]$. Using the triangle inequality 
	   \[    \| x_{i,k+1} - x^*\|  \leq \| x_{i,k+1} - \bar{x}_{k+1} \| + \| \bar{x}_{k+1} - x^*\|, \]
	   and Lemma III.3, 
	    we can show $\| x_{i,k} - x^*\|$ also converges linearly. Meanwhile, we need to carefully consider the relation between $\alpha_k$ and the network  and problem parameters. The proof is provided in the Appendix. 
\begin{theorem}\label{thm:linear_rate} 
	Let $N\geq 2$ and $x^*$ be a local sharp minimizer of problem \eqref{opt:pro_wcvx}.  
	Suppose the initial points $x_{1,0},\ldots,x_{N,0}$ in Algorithm \ref{alg:proj_sub} satisfy for all $ i\in\{1,\ldots,N\}$ the three constraints
	\begingroup\allowdisplaybreaks
	\begin{align*}
	   \sum_{i=1}^N\|{x}_{i,0} - x^*\|^2&\leq \frac{N}{\Gamma^2} \min\left\{ (\frac{2\beta}{\rho})^2,B^2 \right\}\\
	   \|x_{i,0} - x^*\|^2&\leq  \frac{\Gamma^2}{N}\sum_{i=1}^N\|{x}_{i,0} - x^*\|^2,  \\
	   \|\Delta_0\| &< \frac{\frac{2} {\Gamma} \beta e_0 - \rho e_0^2}{2(L+\beta) c}\lambda,
	\end{align*}\endgroup
	where $c,\lambda$ are constants given in \cref{lem:linear_rate_of_Phi}. Under \cref{assup:graph1,assup:graph2,assup:graph4,assup:graph3,assup:bound_subgrad,assup:geo_step},
	there exists sufficiently small $\delta>0$ such that  
	 for $\gamma =\lambda^{\delta}$,
	we have 
	\be\label{ineq:linear_1}
	\sum_{i=1}^{N}\| {x}_{i,k}-x^*\|^2 \leq  N \gamma^{2k}e_0^2
	\ee
	and
	\be\label{ineq:linear_reg2} \| {x}_{i,k}-x^*\|^2 \leq \Gamma^2 \gamma^{2k} e_0^2\ee for any sequence $\{x_{i,k}\}$ generated by Algorithm \ref{alg:proj_sub}. 
\end{theorem}
The following comments about the theorem are in order:
 \begin{enumerate}
 	\item [(1)] The convergence rate $\gamma=\lambda^\delta$ is the same as the decaying rate of stepsize. But it cannot be smaller than $\lambda$, which is the convergence rate of the consensus. 
 	\item[(2)] For the centralized subgradient method\cite{davis2018subgradient}, the local linear rate is established in the tube 
\[ \mathcal{T}=\{ x: \dist(x,\cX_*)\leq \frac{2\beta}{\rho} \},\]
where $\cX_*$ is the set of the  local sharp minimizers. In \cref{thm:linear_rate}, the initialization constraints ensure that the individual initial points are close enough to each other as well as a sharp minimizer (local convergence). Moreover, since we use the local sharpness property, the local region should be included in $\mathcal{B}$.   
 	\item [(3)] An immediate corollary of \cref{thm:linear_rate} is that $\|\bar{x}_k-x^*\|^2\leq 1/N\sum_{i=1}^{N}\|x_{i,k}-x^*\|^2\leq  \gamma^{2k}e_0^2$ under the same conditions.
 	\item [(4)] If $f_i(x)$ is convex, i.e., $\rho=0$, then the condition   $e_0\leq \frac{\beta}{\Gamma\rho}$   can be removed. The weak convexity parameter $\rho$   restricts the initialization region, which is also clearly stated for centralized subgradient method\cite{davis2018subgradient}. 
 \end{enumerate}

\subsection{Distributed Projected Stochastic Subgradient Method}
In some problems, the function $f_i(x)$ at local agent is given by 
$f_i(x)= \frac{1}{m_i}\sum_{j=1}^{m_i} f_{i,j}$, where $m_i$ is a large number and  $f_{i,l}$ is $\rho-$weakly convex.  Therefore, it is expensive to compute the subgradient of $f_i(x)$ in each iteration. In contrast to the algorithm \eqref{alg:dis_subgrad_orgin},  the distributed  stochastic projected subgradient method iterates as follows
\be\label{alg:sto_proj_sub}  x_{i,k+1} = \Proj_{\mathcal{X}}\left( v_{i,k} - \alpha_k \xi_{i,k}\right),   \ee
where $\alpha_k>0$ is the stepsize,    
and  $\xi_{i,k}$ satisfies $\E \xi_{i,k} \in \partial f_i(v_{i,k}).$ In practice, for each $i$, we uniformly randomly select index $i_l\in\{1,2,\ldots,m_i\}$ and set $\xi_{i,k} \in \partial f_{i_l}(v_{i,k})$. 
We also assume that $\E \normtwo{\xi_{i,k} }^2 \leq L^2$ for all $i,k$, which is standard as in \cite{davis2019stochastic}. 
We have the following convergence result for distributed projected stochastic subgradient method \eqref{alg:sto_proj_sub}. The proof is given  in the Appendix. 
\begin{theorem}\label{thm:convergence_stoDPSM}
	Let $t<\frac{1}{2\rho}$ and  $\{ x_{i,k}\}$ be the sequence of algorithm \eqref{alg:sto_proj_sub}.
	Under  \cref{assup:graph1,assup:graph2,assup:graph4,assup:graph3,assup:bound_subgrad,assup:stepsize}, 
	\[\lim\limits _{T\rightarrow \infty} \inf_{k\leq T}\E\|\nabla \varphi_t(\bar{x}_k)\| = 0.\]
\end{theorem}

\section{Numerical Experiment}
We conduct simulations on robust phase retrieval problem 
\be\label{prob:robust_ph_r} \min_{x\in\R^n} f(x) = \frac{1}{N} \sum_{i=1}^N (\frac{1}{m} \sum_{j=1}^{m} \abs{\inp{w_{i,j}}{x} ^2-y_{i,j} }). \ee
The problem is to recover the random signal $\tilde{x}\in\R^n$ using  Gaussian measurements $w_{i,j}$. 
We generate the measurements $w_{i,j}$ and the observations  $y_{i,j}$ following the work \cite{duchi2019solving}. For simplicity, we only consider the noiseless case.  More specifically, the ground truth $\tilde{x}$ is drawn from $N(0,I_n)$ and   $y_{i,j}=\inp{w_{i,j}}{\tilde{x}}^2$, where $w_{i,j}$  are i.i.d standard Gaussian random variables.  As suggested by \cite{duchi2019solving}, the recovery rate is $100\%$ when $N\times m\geq 2.7n$  for the proximal linear algorithm. Therefore, we  use  $N\times m \geq 3n$ for subgradient method in all tests. The initialization follows from  the procedure proposed in \cite[Section 4.2]{duchi2019solving} and we set $x_{1,0}=x_{2,0}=\ldots=x_{N,0}$. 

\begin{figure}
	\centering\includegraphics[width=0.8\linewidth]{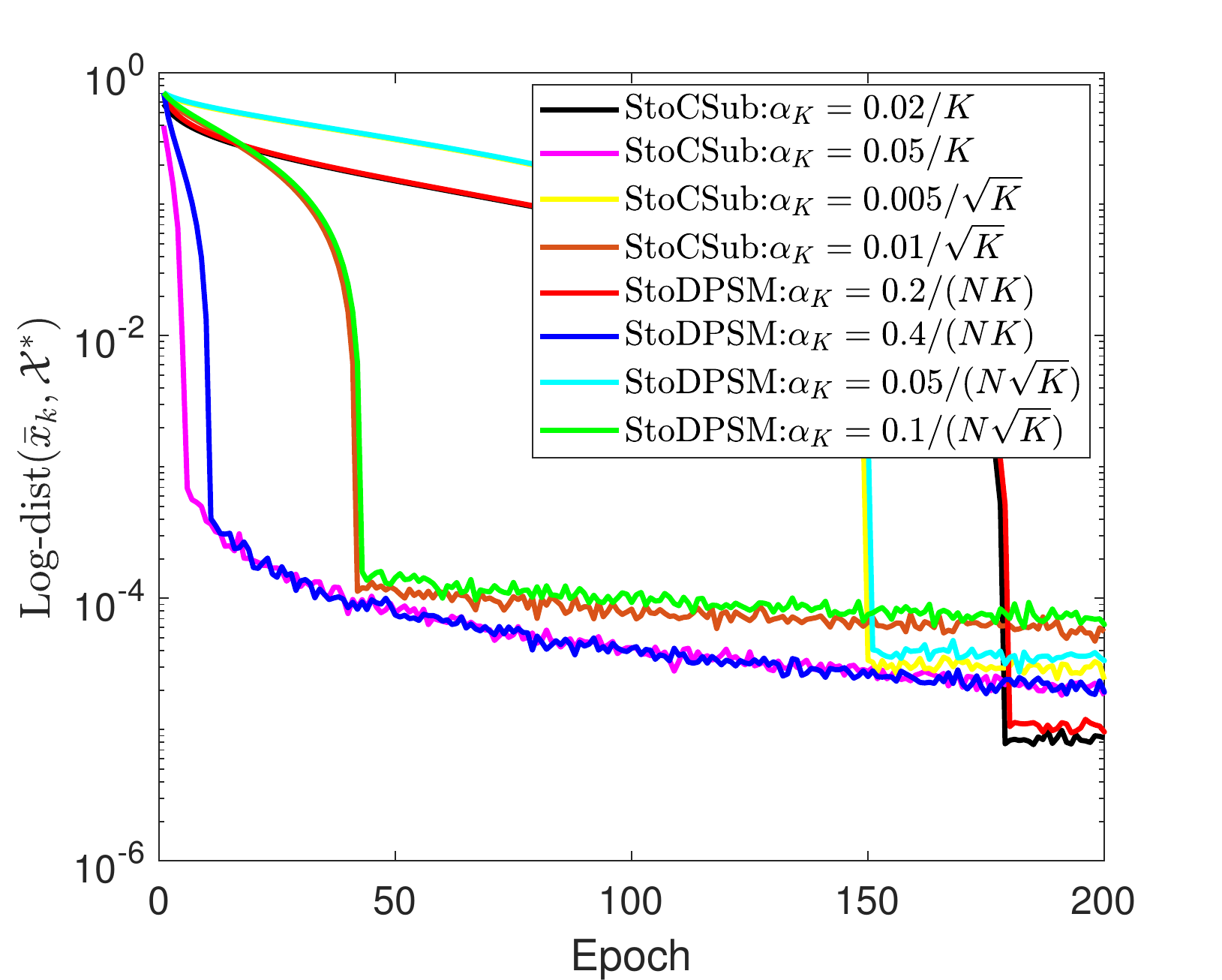}
	\caption{ Stochastic DPSM. $n=100, N = 10, m =1000$. }\label{fig:sto_DPSM}
\end{figure}
 
 The robust phase retrieval formulation \eqref{prob:robust_ph_r} was shown to be weakly convex \cite{duchi2019solving} and have sharpness property w.h.p under mild probabilistic assumptions in \cite{eldar2014phase,duchi2019solving}.
  Different from the  \cref{def:sharpness}, the sharpness condition is given by  
 \[ f(x) - \min f \geq \kappa \|x-\tilde{x}\| \|x+\tilde{x}\|, \]
 where $\kappa>0$ is some number. Hence,   $\pm\tilde{x}$ are also the global minimizers.
 
The global minimizers set is $\{\tilde{x},-\tilde{x}\}$. 
According to \cite[Lemma 3.1]{davis2017nonsmooth}, there is no other critical points in the tube $\{x:\dist(x,{\cX}^*)\leq\frac{2\beta}{\rho}\}$. Since $0$ is also a critical point to the population function $f_P(x)=\E_a[|\inp{a}{x}^2 - \inp{a}{\tilde{x}}^2|]$\cite[Theorem 5.1]{davis2017nonsmooth}. We have $\{x:\|x-\tilde{x}\|\leq \frac{2\beta}{\rho}\}\cap \{x:\|x+\tilde{x}\|\leq\frac{2\beta}{\rho}\}=\emptyset$. To satisfy \cref{def:sharpness}, we let $\beta = \kappa\|\tilde{x}\|$ and choose $\mathcal{B} = \{x:\|x-x^*\|\leq \frac{2\beta}{\rho}\}$, where $x^*$ is $\tilde{x}$ or $-\tilde{x}$ and the sign is decided by the initialization. 
 
\textbf{Synthetic data} First, we solve the robust phase retrieval problem \eqref{prob:robust_ph_r} by stochastic DPSM using diminishing stepsize. We generate an Erd\"{o}s-R\'{e}nyi model $\mathsf{ER}(N, 0.3)$ and $A(k)=A$ is time-invariant Metropolis Hasting matrix associated with the graph. Therefore, we have $\lambda$ is the second largest singular value of $A$ in \cref{lem:linear_rate_of_Phi}.  In each epoch $K$, the stepsize is set to $\alpha_K=\mathcal{O}(1/K)$ or $\alpha_K=\mathcal{O}(1/\sqrt{K})$.  We plot the log distance v.s. epoch  $K$ in \cref{fig:sto_DPSM}. We also compare stoDPSM with the stochastic centralized subgradient method(StoCSub)\cite{davis2019stochastic}. To make a fair comparison, we set the mini-batch size in StoCSub to $N=10$. We tune the stepsize such that the best performance is achieved. We see that the convergence of StoDPSM is comparable with StoCSub in the epoch.

\begin{figure*}[ht]
	\begin{center}
		\minipage{0.24\textwidth}
		\subfigure[convergence, $p=0.1$]{
			{\includegraphics[width=0.98\linewidth]{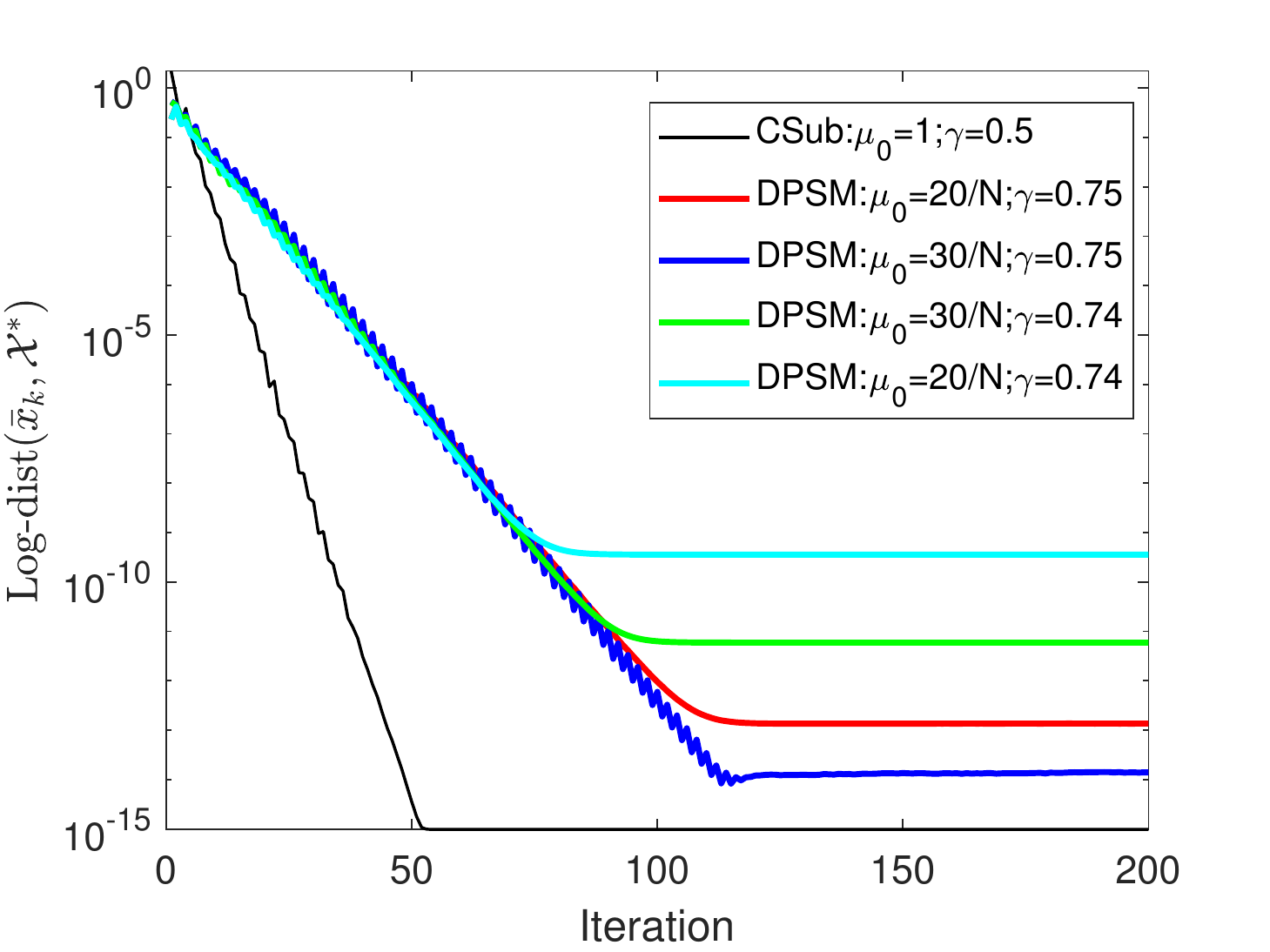}} }
		\endminipage\hfill
		\minipage{0.24\textwidth}
		\subfigure[$\sigma_2(k)$, $p=0.1$]{	
			{\includegraphics[width=0.98\linewidth]{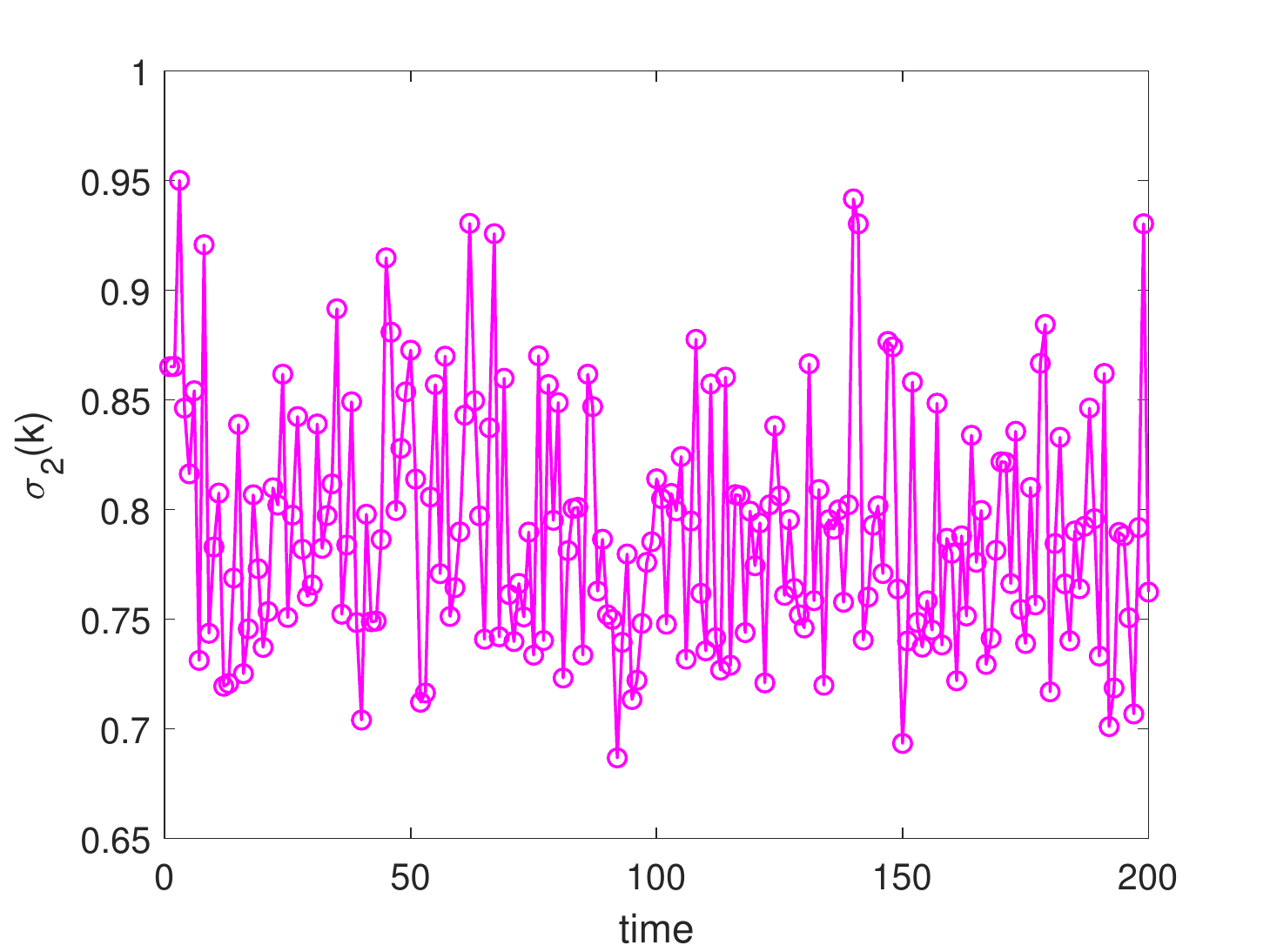}}}
		\endminipage\hfill	
	\minipage{0.24\textwidth}
		\subfigure[convergence, $p=0.2$]{
			{\includegraphics[width=0.98\linewidth]{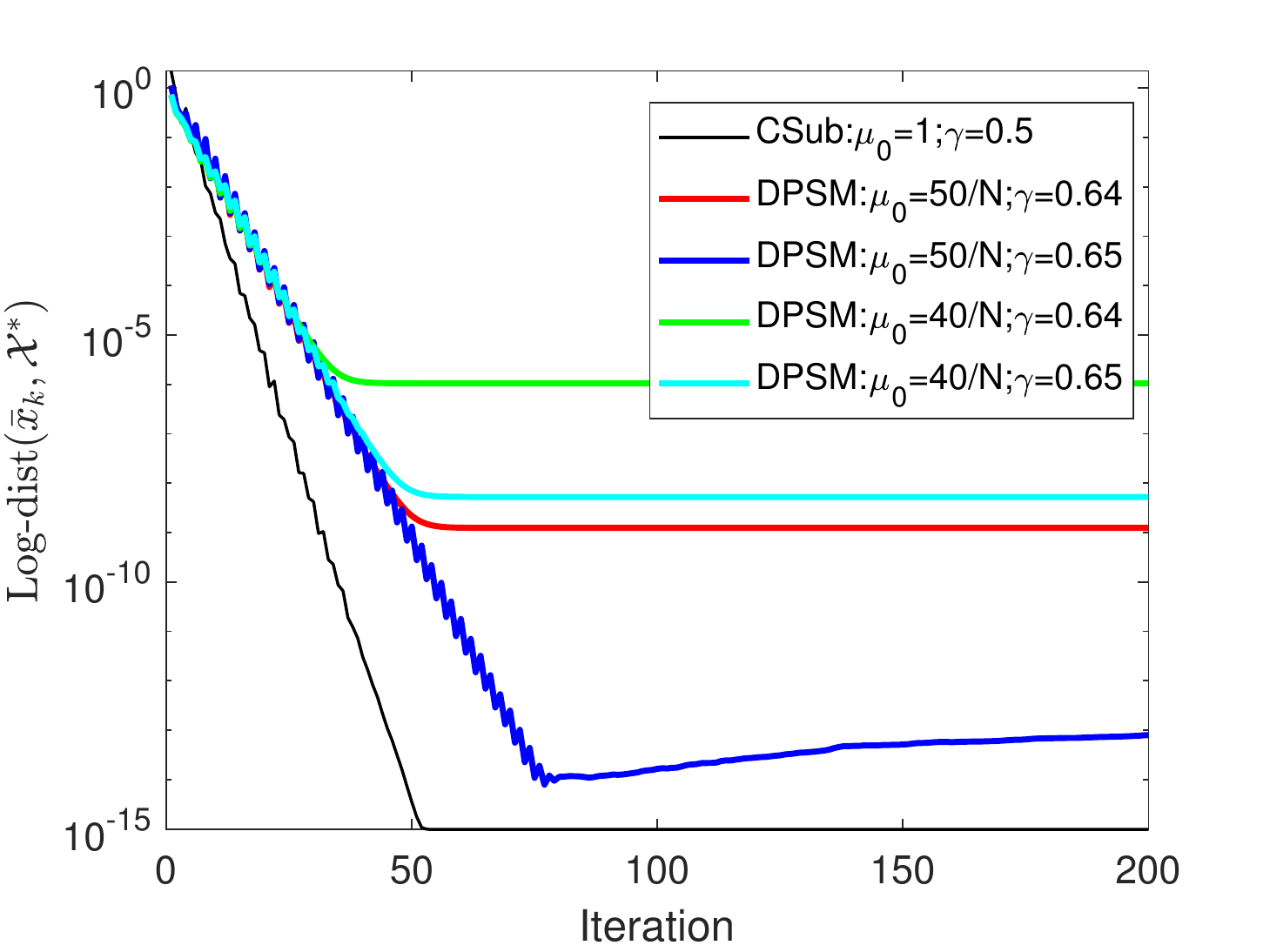}} }
		\endminipage\hfill
		\minipage{0.24\textwidth}
		\subfigure[$\sigma_2(k)$, $p=0.2$]{	
			{\includegraphics[width=0.98\linewidth]{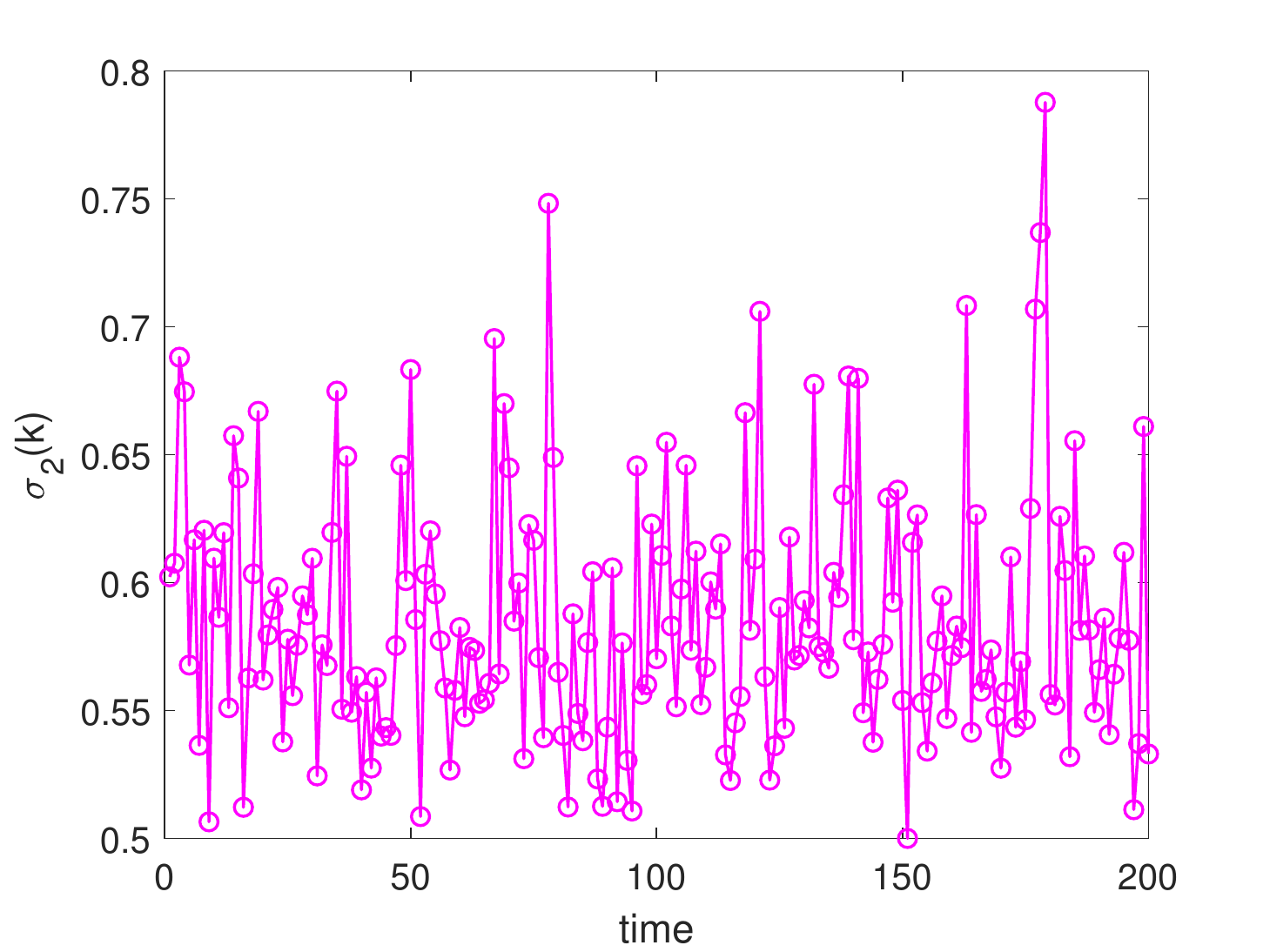}}}
		\endminipage\hfill	
		\caption{Linear rate with different stepsize. Data size: $n=100, N = 100, m =50$.   }\label{fig:linear_syn}
		\label{figure:drgta}
	\end{center}
	\vskip -0.1in
\end{figure*}
 Next, we demonstrate the linear rate of DPSM. The data size is fixed with $N=100, n=100, m = 50$. 
Like the centralized subgradient method\cite{davis2018subgradient}, $\mu_0$ and $\gamma$ should be tuned. In \cref{fig:linear_syn}, `CSub' represents the centralized subgradient method\cite{davis2018subgradient}. The graph is time varying. Specifically, we generate an Erd\"{o}s-R\'{e}nyi model $\mathsf{ER}(100, p)$  and Metropolis Hasting matrix associated with the graph  at each iteration. In \cref{fig:linear_syn} (a) and (b), the probability $p$ is $0.1$.  We demonstrate the linear convergence of CSub and DPSM with different stepsize in \cref{fig:linear_syn} (a). We see that $\gamma=0.5$ works for CSub but not for DPSM, since the smallest $\gamma$ is $0.75$ for $\mu_0=30/N$. For $\mu_0=20/N, \gamma=0.75$, DPSM does not converge to the same precision as $\mu_0=30/N$, so the largest $\mu_0$ may be $30/N$.  This indicates that convergence rate of DPSM is slower than CSub. In \cref{fig:linear_syn} (b), we plot the $\sigma_2(k)$ w.r.t the iteration, where $\sigma_2(k)$ is the second largest singular value of the matrix $A(k)$. We see that most singular values $\sigma_2(k)$ are larger than 0.7.   And $\gamma=0.75>0.7$ also demonstrates that the convergence rate cannot be faster than consensus.   In \cref{fig:linear_syn} (c) and (d), we show the similar results for $p=0.2$. Since the connectivity is stronger, we find that smaller  $\gamma=0.65$ can guarantee the linear rate. Although DPSM is not faster than CSub in the iteration number, DPSM has the advantage of parallel computation. And if the data number $m\times N$ is large, the computation of the whole subgradient is not affordable. We also test the case $m=1$, i.e., there is only single data at each node. In this case, the graph is time-invariant. We fix the graph following $\mathsf{ER}(400,0.3)$. However,  the smallest $\gamma=0.985$ is observed, which is away from   $\lambda = 0.28067$ compared with  \cref{fig:linear_syn}. This could be because single data in local node contributes little information about the sharpness. 
\begin{figure}
	\centering\includegraphics[width=0.8\linewidth]{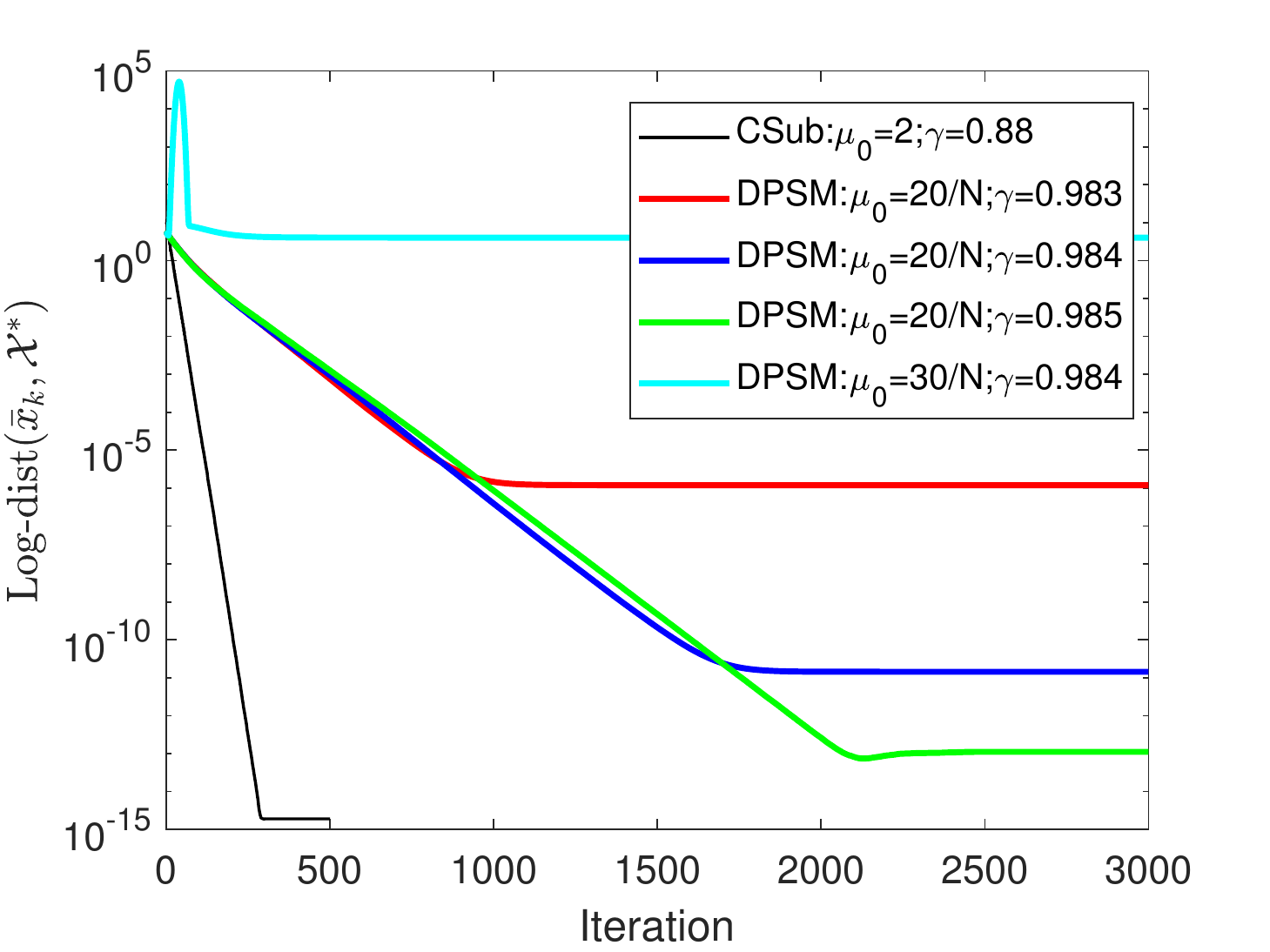}
	\caption{Linear rate with different stepsize. Data size: $n=100, N = 400, m =1$. $\lambda = 0.28067$. }\label{fig:linear_syn1}
\end{figure}

\textbf{Real-world image}
We  use digit images from the MNIST data set\cite{lecun1998gradient}. The gray image dimension is $n=28\times 28=784$ and we set $m=84, N=28$ so that the number of Gaussian measurements is $m\times N = 3\times n$. Other settings are the same as previous synthetic data. {We fix a graph following $\mathsf{ER}(28,0.3)$}. In \cref{fig:MINIST}, we show the original, initial guess and the recovered image. We see that the recovery is identical to the true image. The convergence plot of DPSM and stoDPSM is shown in \cref{fig:MINIST1}.

\begin{figure}
	\centering\includegraphics[width=0.99\linewidth]{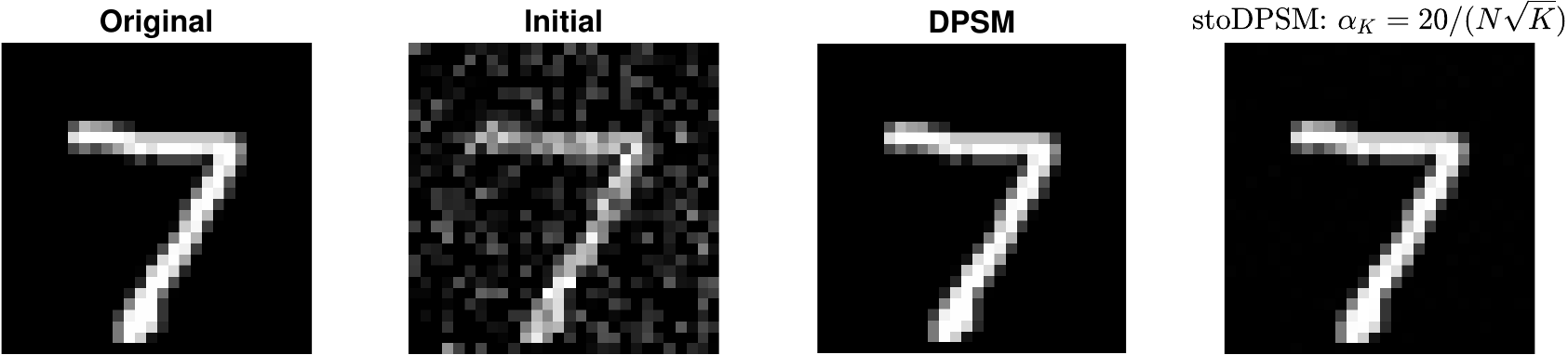}
	\caption{Digit recovery; the first one is the true digit, the second is the initial guess, the third is the digit produced by DPSM and the last is produced by stoDPSM with $\alpha_K=20/(N\sqrt{K})$. Data size: $n=784, N = 28, m =84$.  }\label{fig:MINIST}
\end{figure}

\begin{figure}
		\minipage{0.24\textwidth}
		\subfigure[DPSM]{
			{	\centering\includegraphics[width=0.95\linewidth]{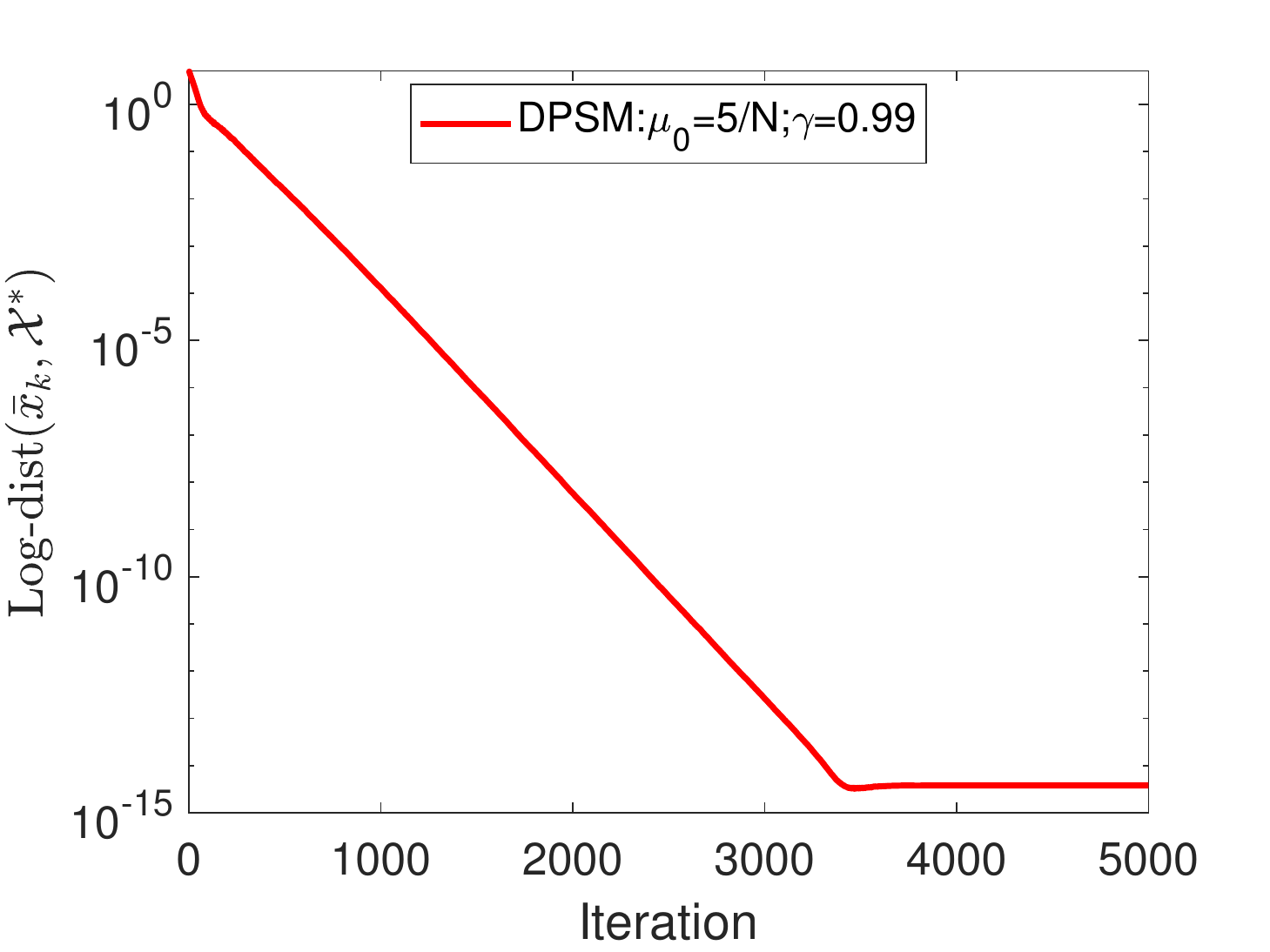} }}
		\endminipage\hfill
		\minipage{0.24\textwidth}
		\subfigure[stoDPSM]{	
			{\includegraphics[width=0.95\linewidth]{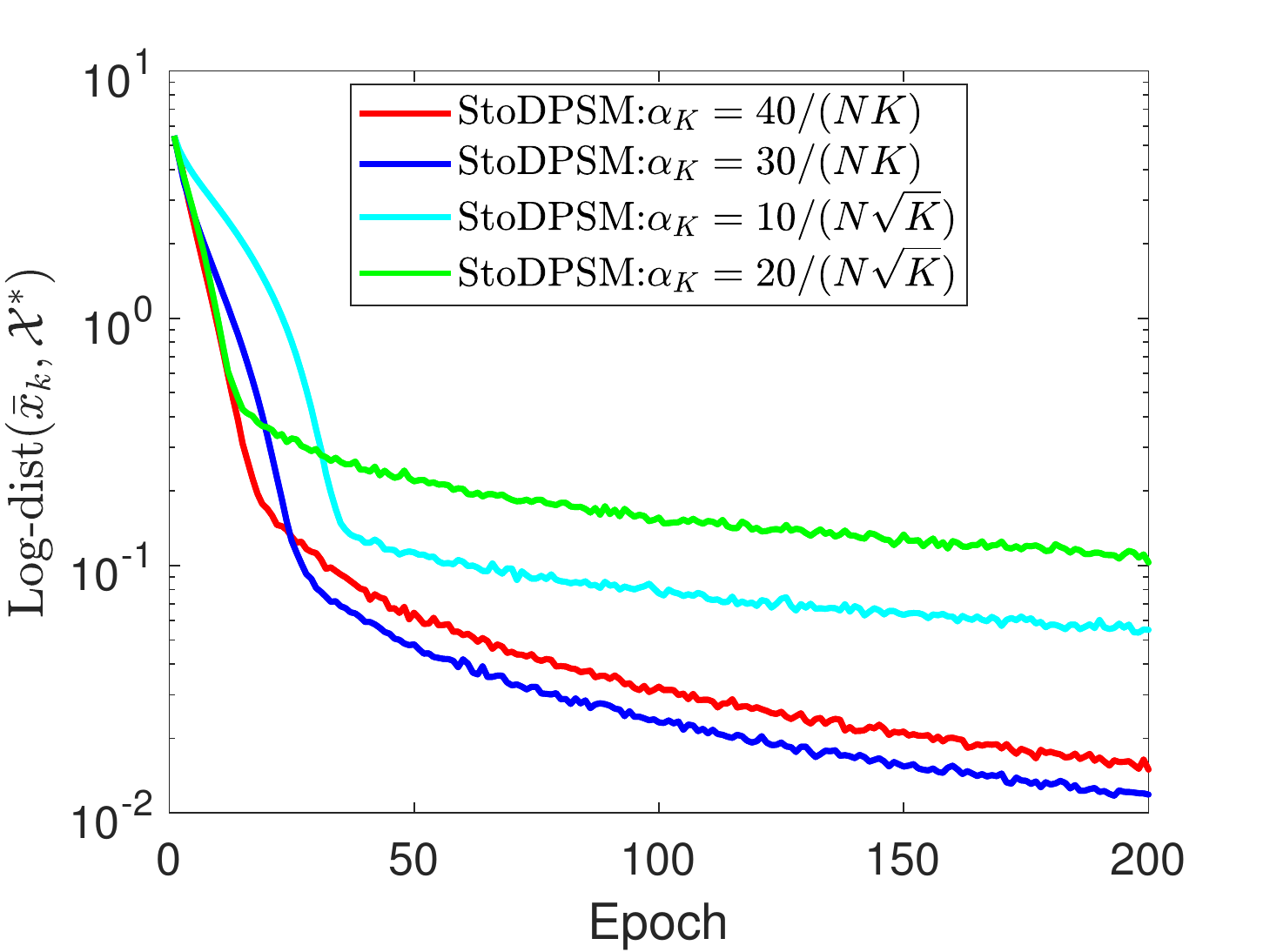}}}
		\endminipage\hfill	
	\caption{Linear rate of DPSM. MINIST Data: $n=784, N = 28, m =84$.  }\label{fig:MINIST1}
\end{figure}

\section{Conclusion}
We analyzed the distributed subgradient method for solving constrained weakly convex optimization. We presented the global convergence of the average point using the notion of Moreau envelope. Moreover, we proved a linear convergence rate under the sharpness property. Numerical results on robust phase retrieval illustrate our theory. 
 
A natural extension of this work is to consider the directed network. For example, the convergence of directed distributed subgradient method for convex problems was analyzed in \cite{xi2016distributed}. It will also be interesting to see whether it is possible to deal with  different constraints at each local node (see e.g., the convex constraints in \cite{nedic2010constrained}). Finally, it will be worth considering non-convex constraints (e.g., sphere constraint \cite{li2019nonsmooth}).


\appendix
 \begin{proof}[\textbf{Proof of \cref{lem:alt_weak_cvx}}]
 	We prove it by induction. For $m=2$, let $y=a_1x_1+a_2x_2$, where $a_1+a_2=1, a_1\geq 0$ and $ a_2\geq 0$. From the subgradient inequality \eqref{ineq:subg_weak_cvx}, we have
 	\[ f(x_1) \geq f(y) + \inp{\partial f(y)}{x_1-y} - \frac{\rho}{2}\|x_1-y\|^2 \]
 	and 
 	\[ f(x_2) \geq f(y) + \inp{\partial f(y)}{x_2-y} - \frac{\rho}{2}\|x_2-y\|^2.\]
 	Multiplying the two above inequalities by $a_1$ and $a_2$, respectively, and summing them, yield 
 	\[ a_1f(x_1) + a_2 f(x_2) \geq f(y) - \frac{\rho}{2} a_1a_2\|x_1-x_2\|^2.\]
 	Similarly, we also have 
 	\[ a_1 g(x_1) + a_2 g(x_2) \geq g(y) + \frac{\tau}{2} a_1a_2\|x_1-x_2\|^2.\]
 	Therefore, inequality \eqref{ineq:alt_def_weak_cvx} holds for $m=2$. Suppose they hold for $m=k$. For $m=k+1$,  let $z = \sum_{i=1}^{k+1} a_i x_i$ and $b = \sum_{i=1}^{k} a_i $. We have
 		\begingroup
	\allowdisplaybreaks
    \begin{align}
 	f(z) &=  f (b\sum_{i=1}^{k} \frac{a_i}{b}x_i+ a_{k+1}x_{k+1}) \notag\\
 	&\leq b f(\sum_{i=1}^{k} \frac{a_i}{b}x_i ) + a_{k+1}f(x_{k+1}) +\notag
 	\\ &\quad \frac{\rho}{2}  a_{k+1}b \| \sum_{i=1}^{k} \frac{a_i}{b}(x_i- x_{k+1})\|^2 \notag\\
 	&\leq b\left( \sum_{i=1}^k\frac{a_i}{b} f(x_i) + \frac{\rho}{2 }\sum_{i=1}^{k-1} \sum_{j= i+1}^k \frac{a_i a_j}{b^2} \|x_i-x_j\|^2\right)\notag \\
 	&\quad + a_{k+1}f(x_{k+1}) + \frac{\rho}{2}  a_{k+1}b \| \sum_{i=1}^{k} \frac{a_i}{b}(x_i- x_{k+1})\|^2,\label{ineq:derive_lem_wcvx}
    \end{align}
 	\endgroup
 	where the first inequality follows from  $b+a_{k+1}=1$ and the second  from assumption step.
 	Notice that since $\|\cdot\|^2$ is $2-$strongly convex, it follows from the assumption for strongly convex function that 
 	\[ 
 	\bad
 	&\|\sum_{i=1}^{k} \frac{a_i}{b}(x_i- x_{k+1})\|^2\\
 	& \leq \sum_{i=1}^{k} \frac{a_i}{b}\|x_i-x_{k+1}\|^2 - \sum_{i=1}^{k-1}\sum_{j=i+1}^k \frac{a_i a_j}{b^2}\|x_i- x_j\|^2.
 	\ead 
 	\]
 	Substituting it into \eqref{ineq:derive_lem_wcvx} yields
 	\[
 	\bad
 	f(z) 
 	&\leq  \sum_{i=1}^k {a_i}  f(x_i) \\
 	&\quad+ \frac{\rho}{2 }\sum_{i=1}^{k-1} \sum_{j= i+1}^k (\frac{a_i a_j}{b} - \frac{a_{k+1} a_i a_j}{b})\|x_i-x_j\|^2 \\
 	&\quad + a_{k+1}f(x_{k+1}) + \frac{\rho}{2}  \sum_{i=1}^{k} {a_i}a_{k+1}\| x_i- x_{k+1}\|^2\\
 	& =  \sum_{i=1}^{k+1} {a_i}  f(x_i) +\frac{\rho}{2 }\sum_{i=1}^{k} \sum_{j= i+1}^{k+1} a_ia_j\|x_i-x_j\|^2,
 	\ead
 	\]
 	where we use $a_{k+1} = 1-b$ in the   equality.
 	Therefore, inequality \eqref{ineq:alt_def_weak_cvx} holds for $m=k+1$. Using the same argument and noticing that $-\|\cdot\|^2$ is $2-$weakly convex, we have that \eqref{ineq:alt_def_strongly_cvx} also holds for $m=k+1$. Hence, we obtain  the desired results.
 \end{proof}

\begin{proof}[\textbf{Proof of  \cref{lem:linear_rate_of_Phi}}]
	It is  shown  in \cite[Proposition 1]{nedic2009distributed}  that
	there exist  $\eta$ such that 
	\[ \left| [\Phi(k,s)]_j^i - \frac{1}{N}\right| \leq  2\frac{1+\eta^{-B_0}}{1-\eta^{B_0}}(1-\eta^{B_0})^{(k-s)/B_0}\]
	for all $s$ and $k$ with $k\geq s$,
	where $[\Phi(k,s)]_j^i -$ denotes the $i-$th row and  $j-$th column element of $\Phi(k,s)$,  $B_0 = (N-1)B$ and $B$ is the intercommunication interval bound of \cref{assup:graph4}.  By using the matrix norm inequality
	\[ \|A\|_{\textit{op}}\leq \|A\|_F\leq N\|A\|_{\infty}  \]
	for any symmetric real matrix $A\in\R^{N\times N}$, where  $\|A\|_F$ is the Frobenius norm and  $\|A\|_{\infty}=\max_{i,j}\abs{[A]_j^i}$,
	we have the desired result, where $c=2N\frac{1+\eta^{-B_0}}{1-\eta^{B_0}}$ and $\lambda=(1-\eta^{B_0})^{B_0^{-1}}$.
\end{proof}

\begin{proof}[\textbf{Proof of  \cref{lem:finite_sum}}] This can be proved following the similar argument of \cite[Lemma 7]{nedic2010constrained}. 
	 Since $\lim_{T\rightarrow \infty}\gamma_T=0$, there exists $M>0$ such that $\gamma_k$ is uniformly bounded, i.e., $\gamma_k\leq M, \forall k\geq 0$.
	For each $T$,  we have $\lambda^{k}\leq \gamma_{T-1}$ for any $k\geq K_0(T):= \lceil \frac{\log \gamma_{T-1}}{\log \lambda} \rceil$. 
	It follows that
	\[ \bad
	&\quad \sum_{k=0}^{T-1}  \lambda^k\gamma_{T-k-1} \\
	&= \sum_{k=0}^{K_0(T)-1}  \lambda^k\gamma_{T-k-1} +\sum_{k=K_0(T)}^{T-1}  \lambda^k\gamma_{T-k-1}\\
	&\leq\frac{1}{1-\lambda} \cdot \max_{0\leq k\leq K_0(T)-1} \gamma_{T-k-1} +  \frac{\lambda^{K_0(T)}}{1-\lambda}\cdot M\\
	&\leq \frac{1}{1-\lambda} \left( \max_{0\leq k\leq K_0(T)-1} \gamma_{T-k-1} + M \gamma_{T-1}\right).
	\ead\]
Recall $\lim_{T\rightarrow \infty}\gamma_T=0$ and $\sum_T\gamma_T=\infty$.	It is clear that  $K_0(T) =  \lceil \frac{\log \gamma_{T-1}}{\log \lambda} \rceil = {o}(T)$, otherwise there exists a subsequence decreasing   geometrically, which contradicts with $\lim_{k\rightarrow \infty} \gamma_{k+1}/\gamma_k = 1$ and $\sum_T\gamma_T=\infty$. Then, we have 
	\[\lim_{T\rightarrow \infty}\frac{\max_{0\leq k\leq K_0(T)-1} \gamma_{T-k-1} }{\gamma_{T-1}} = 1.\]
	Therefore, $\sum_{k=0}^{T-1}  \lambda^k\gamma_{T-k-1} = \mathcal{O}(\frac{\gamma_{T-1}}{1-\lambda})$ holds for sufficiently large $T$ and thus we have the desired result.
\end{proof}

\begin{proof}[Proof of \cref{lemma:consensus1} and \cref{lem:consensus}]
	The      inequality \eqref{lem:consensun_ineq1} is the same as \cite[Lemma 8]{nedic2010constrained}. We provide the proof for completeness. 
	Without loss of generality, we assume $n=1$. Define 
	\be 
	\bad
	x_k&=[x_{1,k},x_{2,k},\ldots,x_{N,k}],\\
	v_k&=[v_{1,k},v_{2,k},\ldots,v_{N,k}],\\
	e_k&=[e_{1,k},e_{2,k},\ldots,e_{N,k}],
	\ead
	\ee
	where $e_{i,k}=  \Proj_{\mathcal{X}}(  v_{i,k} - \alpha_k g_{i,k} )- v_{i,k}$.
	The iteration \eqref{alg:proj_sub} can be  rewritten as
	\begin{equation}\label{alg:proj_sub_re}
	x_{k+1} = v_k + e_k = A(k)x_k + e_k.
	\end{equation}
	That is, the iteration is split into a linear term  $A(k)x_k$ and a nonlinear term $e_k$. Using \cref{lem:property_projection} and \cref{assup:bound_subgrad}, it follows that 
	\begin{equation}\label{ineq:bound_e_k}
	\bad
	\|e_{i,k}\|^2&\leq  \| v_{i,k} - \alpha_k g_{i,k} - v_{i,k}\|^2\leq \alpha_k^2 L^2.
	\ead
	\end{equation}
	Therefore, we have
	\begin{equation}\label{ineq:bound_e}
	\| e_k\| \leq \sqrt{N} L\alpha_k.
	\end{equation}
	Let $J = \frac{1}{N}\mathbf{11}^T,$ where $\mathbf{1}\in\R^N$ is a column vector with all elements $1$. Then, $\Delta_k = x_k - Jx_k$. We have
	\begin{equation}\label{ineq_derive_error_to_averge}
	\bad
	\Delta_{k+1} &= (I-J)x_{k+1} \\
	&= (I-J) A(k)x_k + (I-J)e_k\\
	&=A(k)x_k - A(k)Jx_k +(I-J)e_k\\
	& = A(k)\Delta_k + (I-J)e_k,
	\ead
	\end{equation}
	where the third equality is due to $JA(k)=J=A(k)J$. 
	Therefore,  the following recursion holds for $k\geq s \geq 0$
	\[   \Delta_{k+1}  = \Phi(k,s)\Delta_{s}+ \sum_{l=s}^{k-1} \Phi(k,l+1)(I-J)e_{l} + (I-J)e_k.\]

	Since $\mathbf{1}^\top \Delta_l= \mathbf{1}^\top (I-J)e_l=0,\ \forall l$, 
	we have
	\[   \bad
\Delta_{k+1} &=( \Phi(k,s)-J)\Delta_{s}\\
&\quad + \sum_{l=s}^{k-1} (\Phi(k,l+1)-J)(I-J)e_{l} + (I-J)e_k.
	\ead \]
It follows from \cref{lem:linear_rate_of_Phi} that there exist $c>0$ and $\lambda\in(0,1)$, where $\lambda$ is independent of $k$, such that
	\be\label{ineq:estimation_of_Delta}
	\bad
	\|\Delta_{k+1}\|& \leq   c\lambda^{k}\|\Delta_0\| +  c\sqrt{N}L \sum_{l=0}^{k-1} \lambda^{k-l-1} \alpha_{l} + \sqrt{N}L\alpha_k.
	\ead\ee
    With the  \cref{lem:finite_sum} and $\lim_{k\rightarrow\infty} \alpha_{k+1}/\alpha_k = 1$, we have \eqref{lem:consensus_order}  as desired.
\end{proof}


\begin{proof}[\textbf{Proof of \cref{lem:one-step}}]
	The following inequality holds because of  the non-expansiveness of  the projector
	\[
	\bad
	\left\|x_{i , k+1}- 	\hat{v}_{i,k} \right\|^{2} &= \normtwo{\Proj_{\mathcal{X}}\left( v_{i,k} - \alpha_k g_{i,k}\right)- \hat{v}_{i,k} }^2 \\
	&\leq \normtwo{ v_{i,k} - \alpha_k g_{i,k} - \hat{v}_{i,k}  }^2\\
	& =  \normtwo{ v_{i,k} - \hat{v}_{i,k}  }^2 - 2\alpha_k\inp{v_{i,k} - \hat{v}_{i,k}  }{g_{i,k}} \\
	&\quad + \alpha_k^2\normtwo{g_{i,k}}^2.
	\ead
	\]
	Recall the weak convexity of $f_i$ and the boundedness of $g_{i,k}$. It follows that 
	\be\label{eq_derivation2}
	\bad
	&\quad \|x_{i , k+1}- 	\hat{v}_{i,k} \|^{2}\\& \leq \normtwo{ v_{i,k} - \hat{v}_{i,k}  }^2  + 2\alpha_k( f_i(\hat{v}_{i,k})- f_i(v_{i,k}) \\
	&\quad + \frac{\rho}{2}\|v_{i,k}-\hat{v}_{i,k}\|^2 ) + L^2\alpha_k^2 .
	\ead \ee

	Using the Lipschitz continuity of $f_i$ and \Cref{lem:nonexpansive:proximal_map}, we have
	\begin{equation}\label{eq_derivation3-1}
	\bad
	&\quad f_i(\hat{v}_{i,k}) - f_i(v_{i,k}) \\
	&= f_i(\hat{v}_{i,k}) -f_i(s_k)+ f_i(s_k) - f_i(\bar{x}_k) + f_i(\bar{x}_k) -   f_i(v_{i,k}) \\
	&\leq L \normtwo{ \hat{v}_{i,k} - s_k} + f_i(s_k) - f_i(\bar{x}_k)  + f_i(\bar{x}_k) -   f_i(v_{i,k}) \\
	&\leq L( \frac{1}{1-t\rho}+1)\normtwo{ v_{i,k} - \bar{x}_k  }+ f_i(s_k) - f_i(\bar{x}_k)  \\
	&\leq \frac{L(2-t\rho)}{1-t\rho}\sum_{j=1}^N a_{i,j}(k)\normtwo{x_{j,k} - \bar{x}_k  }+ f_i(s_k) - f_i(\bar{x}_k) 
	\ead
	\end{equation}
	and 
	\begin{equation}\label{eq_derivation3-2}
	\bad
	&\quad \frac{\rho}{2}\normtwo{v_{i,k} - \hat{v}_{i,k}   }^2 \\
	&=\frac{\rho}{2}\normtwo{ v_{i,k}  -\bar{x}_k +\bar{x}_k -s_k + s_k -\hat{v}_{i,k}  }^2\\
	&\leq \rho\normtwo{\bar{x}_k -s_k  }^2 + \rho\normtwo{ v_{i,k}  -\bar{x}_k + s_k -\hat{v}_{i,k} }^2\\
	&\leq \rho\normtwo{\bar{x}_k -s_k  }^2 +  2\rho(1+ \frac{1}{(1-t\rho)^2})\normtwo{ v_{i,k}  -\bar{x}_k  }^2\\
	&\leq \rho\normtwo{\bar{x}_k -s_k  }^2 \\
	&\quad+  2\rho(1+ \frac{1}{(1-t\rho)^2})\sum_{j=1}^N a_{i,j}(k)\normtwo{ x_{j,k}  -\bar{x}_k  }^2.
	\ead
	\end{equation}
	
	Summing inequalities \eqref{eq_derivation3-1} and \eqref{eq_derivation3-2}  for $i=1,\ldots,N$, yields
	\be \label{eq_derivation4-1}
	\bad
	&\sum_{i=1}^N \left( f_i(\hat{v}_{i,k}) - f_i(v_{i,k}) +\frac{\rho}{2} \normtwo{v_{i,k} - \hat{v}_{i,k}   }^2   \right)	\\
	\leq  & \frac{L(2-t\rho)}{1-t\rho} \sum_{i=1}^N \normtwo{x_{i,k} - \bar{x}_k  }+ N( f(s_k) - f(\bar{x}_k) )\\
	+&N\rho \normtwo{\bar{x}_k -s_k  }^2 +  2\rho(1+ \frac{1}{(1-t\rho)^2})\sum_{i=1}^N\normtwo{ x_{i,k}  -\bar{x}_k  }^2.
	\ead\ee
	From the definition of $s_k$, if $t<\frac{1}{2\rho}$, one has
	\be \label{eq_derivation4-2}
	\bad
	&\quad f(s_k) - f(\bar{x}_k)  +   \rho  \normtwo{\bar{x}_k -s_k  }^2\\
	&= f(s_k) - f(\bar{x}_k)  +  (\frac{1}{2t}-\frac{1}{2t} +\rho) \normtwo{\bar{x}_k -s_k  }^2\\
	&\leq (-\frac{1}{2t} +\rho ) \normtwo{\bar{x}_k -s_k  }^2.
	\ead\ee 
	Therefore, we have \eqref{ineq:one_step} by combining  \eqref{eq_derivation2}, \eqref{eq_derivation4-1} and \eqref{eq_derivation4-2}.
\end{proof}


\begin{proof}[\textbf{Proof of \cref{thm:convergence_1}}]
	
(1).
	From the definition of $\varphi_{t}(x_{i ,k+1})$, we have
	\begin{equation}\label{eq_derivation1}
	\varphi_{t}\left(x_{i ,k+1}\right) \leq f(z)+\frac{1}{2 t}\left\|x_{i , k+1}-z\right\|^{2}, \quad \forall z\in \cX.
	\end{equation}
	Let
	$\hat{v}_{i,k}  =\argmin _{y\in\cX} f(y)+\frac{1}{2 t}\|y-v_{i,k}\|^{2}$ and 
	$\hat{x}_{i,k}  =\argmin _{y\in\cX} f(y)+\frac{1}{2 t}\|y-x_{i,k}\|^{2}.$
	
	Substituting $z = \hat{v}_{i,k}$  into \eqref{eq_derivation1}, we obtain
	\begin{equation}\label{eq_derivation3}
	\bad
	\varphi_{t}\left(x_{i ,k+1}\right) &\leq f(\hat{v}_{i,k})+\frac{1}{2 t}\left\|x_{i , k+1}-\hat{v}_{i,k}\right\|^{2}.
	\ead
	\end{equation}
	Summing the above inequality for $i$ and using inequality \eqref{ineq:one_step}  yields
	\be \label{eq_derivation}
	\bad
	&\sum_{i=1}^N \varphi_{t}\left(x_{i ,k+1}\right) \\
	\leq &\sum_{i=1}^N \varphi_{t}\left(v_{i ,k}\right)
	+\frac{\alpha_k}{t}\left(   N(-\frac{1}{2t} +\rho ) \normtwo{\bar{x}_k -s_k  }^2 \right.\\
	&\quad \left.+ \frac{L(2-t\rho)}{1-t\rho} \sum_{i=1}^N \normtwo{x_{i,k} - \bar{x}_k  } \right.\\
	&\quad \left.+   2\rho(1+ \frac{1}{(1-t\rho)^2})\sum_{i=1}^N\normtwo{ x_{i,k}  -\bar{x}_k  }^2 \right) + \frac{N L^2\alpha_k^2 }{2t}.
	\ead\ee
	Noticing $v_{i,k}=\sum_{j=1}^{N} a_{i,j}(k)x_{j,k}$,  we have 
	\begingroup\allowdisplaybreaks
	\begin{align}
	&\quad\varphi_t(v_{i,k}) \notag \\&= f( \sum_{j=1}^N a_{i,j}(k) \hat{v}_{i,k} )  + \frac{1}{2t}\normtwo{\sum_{j=1}^N a_{i,j}(k) (\hat{v}_{i,k} - x_{j,k} )  }^2 \notag \\
	&\leq  f( \sum_{j=1}^N a_{i,j}(k) \hat{x}_{j,k} )  + \frac{1}{2t}\normtwo{\sum_{j=1}^N a_{i,j}(k) (\hat{x}_{j,k} - x_{j,k} )  }^2\notag  \\
	&\leq \sum_{j=1}^N a_{i,j}(k) f(\hat{x}_{j,k} )\notag  \\
	&\quad + \frac{\rho}{2} \sum_{j=1}^{N-1}\sum_{l= j+1}^N a_{i,j}(k) a_{i,l}(k)\|\hat{x}_{j,k}-\hat{x}_{l,k}\|^2\notag \\
	&\quad + \sum_{j=1}^N a_{i,j}(k) \frac{1}{2t}\normtwo{\hat{x}_{j,k} - x_{j,k} }^2\notag \\
	& \leq  \sum_{j=1}^N a_{i,j}(k) \varphi_t(x_{j,k})\notag \\
	&\quad +\frac{\rho}{2(1-t\rho)^2} \sum_{j=1}^{N-1}\sum_{l= j+1}^N a_{i,j}(k) a_{i,l}(k)\|{x}_{j,k}-{x}_{l,k}\|^2 ,\notag 
	\end{align}\endgroup
	where the first inequality is because of  the definition of $\hat{v}_{i,k}$ and $\sum_{j=1}^N a_{i,j}(k) \hat{x}_{j,k} \in \cX$, the second inequality follows from inequality \eqref{ineq:alt_def_weak_cvx} in Lemma \ref{lem:alt_weak_cvx} and the convexity of $\|\cdot\|^2$ and the last inequality holds due to Lemma \ref{lem:nonexpansive:proximal_map}.
	Letting $\bar{\varphi}_{t,k+1}:= \frac{1}{N}\sum_{i=1}^{N}\varphi
	_{t}\left( x_{i,k+1}\right) $ together with  \eqref{eq_derivation} gives
	\begingroup\allowdisplaybreaks
	\begin{align}\label{inequality}
	& \bar{\varphi}_{t,k+1}\notag \leq \bar{\varphi}_{t,k}\notag \\
	& +\frac{\rho }{2N(1-t\rho)^{2}}\cdot\notag\\
	&\quad \sum_{i=1}^{N}\sum_{j=1}^{N-1}\sum_{l=j+1}^{N}a_{i,j}(k)a_{i,l}(k)%
	\Vert {x}_{j,k}-{x}_{l,k}\Vert ^{2}  \notag \\
	&+\frac{\alpha _{k}}{t}\left( (\rho-\frac{1}{2t} )\Vert \bar{x} 
	_{k}-s_{k}\Vert ^{2}\right.\notag\\
	& +\frac{L(2-t\rho )}{N(1-t\rho )} 
	\sum_{i=1}^{N}\Vert x_{i,k}-\bar{x}_{k}\Vert  \notag \\
	&+\left. \frac{2\rho }{N}(1+\frac{1}{(1-t\rho )^{2}})\sum_{i=1}^{N}\Vert
	x_{i,k}-\bar{x}_{k}\Vert ^{2}\right) +\frac{L^{2}\alpha _{k}^{2}}{2t}  \notag
	\\
    &	\leq \bar{\varphi}_{t,k}+b_{k}+\frac{L^{2}\alpha _{k}^{2}}{2t},  
	\end{align}
	\endgroup
	where {\small
	\begin{align*}
	\quad &b_{k}\\
	:=& \frac{\rho }{2N(1-t\rho )^{2}}\sum_{i=1}^{N}\sum_{j=1}^{N-1}%
	\sum_{l=j+1}^{N}a_{i,j}(k)a_{i,l}(k)\Vert {x}_{j,k}-{x}_{l,k}\Vert ^{2} \\
	& \quad +\frac{\alpha _{k}}{t}\left( \frac{L(2-t\rho )}{N(1-t\rho )}%
	\sum_{i=1}^{N}\Vert x_{i,k}-\bar{x}_{k}\Vert\right. \\
	& \quad \left. +\frac{2\rho }{N}(1+\frac{1}{(1-t\rho )^{2}}%
	)\sum_{i=1}^{N}\Vert x_{i,k}-\bar{x}_{k}\Vert ^{2}\right) .
	\end{align*}%
	}
	and the last inequality in (\ref{inequality}) follows from $-\frac{1}{2t}%
	+\rho <0$. By invoking \cref{lem:consensus}, we have 
	\be\label{price_bound_1} \sum_{i=1}^{N}  \sum_{j=1}^{N-1}\sum_{l= j+1}^N a_{i,j}(k) a_{i,l}(k)\|{x}_{j,k}-{x}_{l,k}\|^2 ={\mathcal{O}(\frac{NL^2\alpha_k^2}{(1-\lambda)^2})},  \ee
	\be\label{price_bound_2}   \alpha_k  \sum_{i=1}^N \normtwo{x_{i,k} - \bar{x}_k  }={\mathcal{O}(\frac{NL\alpha_k^2}{1-\lambda})},\ee
	 
	\be\label{price_bound_3}    \alpha_k  \sum_{i=1}^N \normtwo{x_{i,k} - \bar{x}_k  }^2 = {\mathcal{O}(\frac{NL^2\alpha_k^3}{(1-\lambda)^2})}
	\ee
	 and thus $ b_{k}={\mathcal{O}(\frac{L^2\alpha_k^2}{(1-\lambda)^2})}$. 
	 Because $f(x)$ is lower bounded on $\cX$, we have $\varphi_t(x)$ is also lower bounded on $\cX$. From \eqref{inequality} it follows that 
	 \[ \bar{\varphi}_{t,k+1} - \inf \varphi_t(x)  \leq \bar{\varphi}_{t,k} - \inf \varphi_t(x)  +  \mathcal{O}(\alpha_k^2).\]
	 Since $\sum_{k=0}^\infty\alpha_k^2<\infty,$ using Lemma 2\footnote{The lemma is stated as follows. Let $u_{k+1}\geq 0$ and let $u_{k+1}\leq (1+\alpha_k)u_k + \beta_k,$ $\sum_{k=0}^{\infty} \alpha_k< \infty$, \ $\sum_{k=0}^{\infty} \beta_k< \infty$. Then $u_k\rightarrow u\geq 0.$} in \cite[Chapter 2.2]{polyak1987introduction} we have $\{\bar{\varphi}_{t,k}\}$ converges to some value $\bar{\varphi}_t$.

	Recall that $\varphi _{t}(x)$ is continuous differentiable. Since $\Vert
	x_{i,k}-\bar{x}_{k}\Vert \rightarrow 0$, it follows that 
	\begin{equation*}
	|\varphi _{t}(x_{i,k})-\varphi _{t}(\bar{x}_{k})|^{2}\rightarrow 0
	\end{equation*}%
	and%
	\begin{equation*}
	\begin{aligned}
	\left\vert \bar{\varphi}_{t,k}-\varphi _{t}(\bar{x}_{k})\right\vert ^{2}&=| 
	\frac{1}{N}\sum_{i=1}^{N}\varphi _{t}(x_{i,k})-\varphi _{t}(\bar{x}
	_{k})|^{2}\\
	&\leq \frac{1}{N}\sum_{i=1}^{N}|\varphi _{t}(x_{i,k})-\varphi
	_{t}(\bar{x}_{k})|^{2}\rightarrow 0.
		\end{aligned}
	\end{equation*}%
	Thus, $\varphi _{t}(\bar{x}_{k})\rightarrow \bar{\varphi}_t$. 
	
	\bigskip
	(2). 	
	 The inequality (\ref{inequality}) can be re-written as%
	\begin{equation}  \label{eq:re-written}
	\frac{\alpha _{k}}{t}(\frac{1}{2t}-\rho )\Vert \bar{x}_{k}-s_{k}\Vert
	^{2}\leq \bar{\varphi}_{t,k}-\bar{\varphi}_{t,k+1}+b_{k}+\frac{L^{2}\alpha
		_{k}^{2}}{2t}.
	\end{equation}

		Using  \eqref{eq:re-written}, we have
		\begin{align*}
		\sum_{k=0}^{\infty}	\frac{\alpha _{k}}{t}(\frac{1}{2t}-\rho )\Vert \bar{x}_{k}-s_{k}\Vert
		^{2}\\
		\leq \bar{\varphi}_{t,0}-  \bar{\varphi}_{t}+\sum_{k=0}^{\infty} b_{k}+\sum_{k=0}^{\infty}\frac{L^{2}\alpha
			_{k}^{2}}{2t}.
		\end{align*}
		Dividing both sides by $\sum_{k=0}^{\infty} \alpha_k$  yields

		  \begin{align*} 
		  &\inf_{k=1,\ldots,\infty} \Vert \nabla \varphi_t(\bar{x}_k)\Vert^2 \\
		  &\leq \frac{2}{1-2t\rho} \frac{\bar{\varphi}_{t,0}- \bar{\varphi}_{t}+\sum_{k=0}^{\infty} b_{k}+\sum_{k=0}^{\infty}\frac{L^{2}\alpha_{k}^{2}}{2t}
		}{ \sum_{k=0}^{\infty} \alpha_k}.
	\end{align*}
		Since $ b_{k}={\mathcal{O}(\frac{L^2\alpha_k^2}{(1-\lambda)^2})}$, if $\alpha_k = \mathcal{O}(1/\sqrt{k})$, for sufficiently large $T$ we have
		\[ \inf_{1\leq t\leq T} \Vert \nabla \varphi_t(\bar{x}_k)\Vert^2 = { \mathcal{O}(\frac{\bar{\varphi}_{t,0}-  \bar{\varphi}_{t}}{\sqrt{T}} + \frac{L^2}{(1-\lambda)^2}\cdot\frac{\log T}{\sqrt{T}}) }.\]

\end{proof}

Before proving \cref{thm:linear_rate}, we need the following technical lemma.
	\begin{lemma}\label{lem:lemma_used_in_local_rate}
	Given $a>0$,  $0< 2b\leq a$ and $c\geq 1$,   the lower bound of the minimum value in   \eqref{prob:lemma} is given by $-\half N a^2 +   \frac{Nba}{c}.$
	\be\label{prob:lemma}
	\bad
	\min_{x_1,\ldots,x_N} &- \half \sum_{i=1}^N(x_i^2 -2 bx_i) \\
	\st &\quad \sum_{i=1}^{N} x_i^2 \leq N a^2, \\ 
	&\quad 0\leq  x_i\leq ca, \quad  \forall i.
	\ead\tag{$P_N$} 
	\ee
\end{lemma}

\begin{proof} [Proof of  \cref{lem:lemma_used_in_local_rate}]
The dual function is given by
\[
 g(\lambda):=\min_{0\leq x_i\leq ca}- \half \sum_{i=1}^N(x_i^2 -2 bx_i) + \lambda(\sum_{i=1}^N x_i^2 - Na^2  ), \]
 where $\lambda\geq 0$.
We have 
\begin{align*}
    &\quad g(\lambda) \\
    &= N\cdot  \min_{0\leq x\leq ca} \{(\lambda-\half) x^2 + bx \}  -  \lambda N a^2 \\
    & = \left\{\begin{matrix}
-\lambda Na^2 & \text{if} \quad \lambda \geq \half-\frac{b}{ca},\\ 
   N\left[ (\lambda - \half)c^2a^2  + cba \right] - \lambda Na^2     & \text{otherwise.}
\end{matrix}\right.
\end{align*}
Note that $\half - \frac{b}{ca}\geq 0.$
Therefore, we have 
\[   \max_{\lambda \geq 0} g(\lambda)=g(\half - \frac{b}{ca})=-\half N a^2 +   \frac{Nba}{c} . \]
The weak duality implies the desired result.
\end{proof}

\begin{proof} [\textbf{Proof of \cref{thm:linear_rate}}]
	 We prove it by induction. 
	By the definition of $e_0$ and the assumptions on $k=0$, we have $\sum_{i=1}^{N}\| {x}_{i,0}-x^*\|^2 \leq  N \gamma^{2k}e_0^2$ and $\|x_{i,0}-x^*\|\leq \Gamma e_0, \forall i\in[N]:=\{1,\ldots,N\}$. Assume that \eqref{ineq:linear_1} and \eqref{ineq:linear_reg2} hold for $k\geq 0$. For $k+1$, we have
	\begingroup
	\allowdisplaybreaks
	\begin{align*}
	&\quad   \sum_{i=1}^N \| x_{i,k+1} - x^*\|^2\\
	& \leq \sum_{i=1}^N \|v_{i,k} - \alpha_k g_{i,k} - x^*\|^2\\
	&\leq \sum_{i=1}^N \left( \|v_{i,k} - x^*\|^2 -2\alpha_k \left\langle v_{i,k} - x^*,g_{i,k}\right\rangle \right) + N L^2\alpha_k^2  \\
	&\leq \sum_{i=1}^N \left( \|v_{i,k} - x^*\|^2 -2\alpha_k(f_i(v_{i,k})- f_i(x^*) )\right.\\
	&\quad \left.+ \alpha_k{\rho}\|v_{i,k}-x^*\|^2  \right) + N L^2\alpha_k^2 \\
	&= \sum_{i=1}^N \left( \|v_{i,k} - x^*\|^2 -2\alpha_k(f_i(v_{i,k})- f_i(\bar{x}_k) \right.\\
	&\quad \left.+ f_i(\bar{x}_k)- f_i(x^*) ) 
	+ \alpha_k{\rho}\|v_{i,k}-x^*\|^2 \right) + N L^2\alpha_k^2 \\
	&\leq \sum_{i=1}^N \left((1+\rho \alpha_k) \|v_{i,k} - x^*\|^2 + 2L\alpha_k\|v_{i,k} - \bar{x}_k\| \right) \\&\quad  -2N\beta\alpha_k \|\bar{x}_k -  x^*\|
	+ N L^2 \alpha_k^2,\\
	\end{align*}
	\endgroup
	where the  third inequality follows from the weak convexity and the last one is due to the sharpness property and Lipschitz continuity of $f_i$. Using the convexity of $\|\cdot\|^2$ and $\|\cdot\|$ and the stochasticity of columns of  $A(k)$, we have
	\begingroup
	\allowdisplaybreaks
	\begin{align}
	&\quad \sum_{i=1}^N \| x_{i,k+1} - x^*\|^2\notag\\
	&\leq \sum_{i=1}^N \left((1+\rho \alpha_k) \|x_{i,k} - x^*\|^2 + 2L\alpha_k\|x_{i,k} - \bar{x}_k\| \right)\notag \\
	&\quad -2N\beta\alpha_k \|\bar{x}_k -  x^*\|+ N L^2 \alpha_k^2\notag\\
	&\leq \sum_{i=1}^N \left((1+\rho \alpha_k) \|x_{i,k} - x^*\|^2 -2\beta\alpha_k \|x_{i,k} -  x^*\|\right) \notag\\
	&\quad + {2(L+\beta) \alpha_k}\sum_{i=1}^{N} \|\bar{x}_{k} - x_{i,k}\|+ N L^2 \alpha_k^2\notag\\
	&\leq  \sum_{i=1}^N \left((1+\rho \mu_0) \|x_{i,k} - x^*\|^2 -2\beta\alpha_k \|x_{i,k} -  x^*\|\right)\notag \\
	&\quad + \frac{2\sqrt{N}(L+\beta) c\|\Delta_0\|}{\lambda} \gamma^k \alpha_k \notag\\
	&\quad + \frac{ 2N (L+\beta)L}{\lambda^2} (\frac{c{\gamma}^{1/\delta-1}}{1-{\gamma}^{1/\delta-1}}+\lambda)\alpha_k^2+N L^2 \alpha_k^2,\label{ineq:linear_rate1} 
	\end{align}
	\endgroup
	where we use 
	$\|\bar{x}_k-x^*\|\geq   \| x_{i,k} - x^*\| -\|\bar{x}_{k} - x_{i,k}\|$ and $\|\cdot\|_1\leq \sqrt{N}\|\cdot\|$ in the second inequality.  
	The last inequality is due to \eqref{ineq:estimation_of_Delta_2}. 
	Recall the induction assumption that $\sum_{i=1}^N \| x_{i,k} - x^*\|^2\leq N e_0^2\gamma^{2k}$ and $\| x_{i,k} - x^*\|\leq \Gamma e_0 \gamma^k$. 
	Since 
	\be\label{ineq:cond_mu_1}\mu_0\leq \frac{e_0}{2\beta - \rho e_0},\ee we have $\frac{2\beta \alpha_k}{1+\rho \mu_0} = \frac{2\beta \mu_0 \gamma^{k}}{1+\rho \mu_0} \leq e_0 \gamma^{k} $. By invoking \cref{lem:lemma_used_in_local_rate}(letting $a=e_0\gamma^k$, $b=\frac{\beta \alpha_k}{1+\rho \mu_0}$ and $c=\Gamma$ in the lemma), 
	we deduce that 
	\begin{align*} 
	& \quad (1+\rho \mu_0)\sum_{i=1}^N \left( \|x_{i,k} - x^*\|^2 - \frac{2\beta\alpha_k}{(1+\rho \mu_0)} \|x_{i,k} -  x^*\|\right)\\
	&\leq  (1+\rho \mu_0)N e_0^2\gamma^{2k} - 2\frac{N}{\Gamma}\beta\alpha_k  e_0 \gamma^k.
	\end{align*} 
	This, together with \eqref{ineq:linear_rate1} yields
		\begingroup
	\allowdisplaybreaks
	\begin{align*} 
	&\quad\sum_{i=1}^N \| x_{i,k+1} - x^*\|^2\\
	&\leq (1+\rho \mu_0) N( e_0 \gamma^{k} )^2 - 2\frac{N}{\Gamma}\beta \mu_0 e_0 \gamma^{2k}\\
	&\quad + \frac{2\sqrt{N}(L+\beta) c\|\Delta_0\|}{\lambda} \mu_0\gamma^{2k}  \\
	&\quad + \frac{2 N(L+\beta) L}{\lambda^2}(\frac{c{\gamma}^{1/\delta}}{1-\gamma^{1/\delta-1}}+\lambda)\mu_0^2  \gamma^{2k} + N L^2 \mu_0^2  \gamma^{2k}\\
	& = N\gamma^{2k} e_0^2 \left( 1 + (\rho  - \frac{2\beta }{\Gamma e_0}+\frac{2(L+\beta)c\|\Delta_0\|}{\sqrt{N}\lambda e_0^2} )\mu_0 + \right.\\
	&\quad + \left.  \frac{\frac{2(L+\beta)L}{\lambda^2}(\frac{c{\gamma}^{1/\delta-1}}{1-\gamma^{1/\delta-1}}+\lambda)+ L^2 }{e_0^2}\mu_0^2 \right)\\
	& =  N\gamma^{2k} e_0^2 \left( 1 - \frac{q}{e_0^2}\mu_0 + \frac{\frac{ac{\gamma}^{1/\delta-1}}{1-\gamma^{1/\delta-1}} + a\lambda+ L^2}{e_0^2}\mu_0^2 \right),
	\end{align*}
	\endgroup
	where $a=\frac{2 (L+\beta)L}{\lambda^2}, q = \frac{2\beta }{\Gamma } e_0- \rho e_0^2-\frac{2(L+\beta) c}{\sqrt{N}\lambda}\|\Delta_0\|$.
	Since $\gamma\in (0,1)$, if we have the following two conditions
	\begin{enumerate}[label=\arabic*)]
		\item 	$q>0$\\
		\item 	\be\label{ineq_induction_step}
		1>\gamma^2 \geq 1-\frac{q\mu_0}{e_0^2} + \frac{\frac{ac\gamma^{1/\delta-1}}{1-\gamma^{1/\delta-1}} + a \lambda+ L^2}{e_0^2}\mu_0^2,\ee 
	\end{enumerate}
	the result follows 
	\[
	\sum_{i=1}^N \| x_{i,k+1} - x^*\|^2\leq N \gamma^{2(k+1)} e_0^2.
	\]

	Proof of Condition 1) Since  $e_0\leq \frac{2\beta}{\rho\Gamma }$ and
	\be\label{ineq:cond_Delta} 
	\|\Delta_0\| < \frac{\frac{2\beta }{\Gamma} e_0 - \rho e_0^2}{2(L+\beta) c}\lambda,  \ee 
	we have $q>  0.$\\
	
	Proof of Condition 2) To ensure \eqref{ineq_induction_step},
	it is sufficient to show \be\label{ineq_induction_step2}
	1>\gamma^2\geq 1-\frac{q\mu_0}{10e_0^2\sqrt{N}} + \frac{\frac{ac\gamma^{1/\delta-1}}{1-\gamma^{1/\delta-1} } + a\lambda + L^2}{e_0^2}\mu_0^2,\ee for some $\gamma \in (0,1)$, which is equivalent to \\
	{\small
	\[\bad-\gamma^{1/\delta+1} + \gamma^2 + (1 -\frac{q\mu_0}{10e_0^2\sqrt{N}} + \frac{ a\lambda + L^2 - ac}{e_0^2}\mu_0^2 )\gamma^{1/\delta-1} \\
	- (1 -\frac{q\mu_0}{10e_0^2\sqrt{N}} + \frac{ a\lambda + L^2}{e_0^2}\mu_0^2 )\geq 0, \ead \]} 
	if we multiply by $(1-\gamma^{1/\delta-1})$ and re-arrange the terms. Consider the function  
	\[ \bad 
	 &\phi(\gamma )\\ =&-\gamma^{1/\delta+1} + (1 -\frac{q\mu_0}{10e_0^2\sqrt{N}} + \frac{ a\lambda+ L^2 - ac}{e_0^2}\mu_0^2 )\gamma^{1/\delta-1} \\
	 &+ \gamma^2 - (1 -\frac{q\mu_0}{10e_0^2\sqrt{N}} + \frac{ a\lambda + L^2}{e_0^2}\mu_0^2 ). 
	\ead
	\] 
	Our goal is to find $1>\delta>0$ such that   $\phi( \lambda^\delta )\geq 0$ when $\mu_0>0$. 
	
	Letting
	$\mathbf{\epsilon} {:=} -\frac{q\mu_0}{10e_0^2\sqrt{N}} + \frac{ a\lambda + L^2}{e_0^2}\mu_0^2$, we have $\epsilon <0$  since $0<\mu_0< \frac{q}{10(a\lambda +L^2)\sqrt{N}}$ due to \cref{assup:geo_step}. By the same token we have
	\be\label{ineq:cond_mu_2} -\frac{ac}{e_0^2}\mu_0^2\geq (\frac{1}{{\Lambda}}-1)\epsilon, \quad \text{as} \quad \mu_0\leq \frac{q}{10\sqrt{N}(a\lambda +L^2+\frac{ ac\Lambda}{1-\Lambda})}. \ee

	Therefore, if $ 0<\mu_0\leq \frac{q}{10\sqrt{N}(a\lambda +L^2+\frac{ ac\Lambda}{1-\Lambda})}$, we have
		\begingroup
	\allowdisplaybreaks
	\begin{align*}
	&\quad\phi(\lambda^\delta )\\
	&=(1-\lambda^{2\delta}+\epsilon - \frac{ac\mu_0^2}{e_0^2})\lambda^{1-\delta}+ \lambda^{2\delta} - (1+\epsilon)\\
	&\geq (1-\lambda^{2\delta}+\frac{1}{\Lambda}\epsilon)\lambda^{1-\delta}+ \lambda^{2\delta} - (1+\epsilon)\\
	&= (1- \lambda^{2\delta})(\lambda^{1-\delta}-1) + \frac{1}{\Lambda}\lambda^{1-\delta} \epsilon - \epsilon.
	\end{align*}
	\endgroup
	It is clear for every $\lambda\in(0,1)$ that 
	\[ (1- \lambda^{2\delta})(\lambda^{1-\delta}-1) + \frac{1}{\Lambda}\lambda^{1-\delta} \epsilon\rightarrow  \frac{1}{\Lambda}\lambda \epsilon \quad \text{as} \quad \delta\rightarrow 0. \]
	Therefore, there exists sufficiently small $\delta>0$ such that $\phi(\lambda^\delta )\geq \frac{1}{\Lambda}\lambda \epsilon-\epsilon>0$, since $\Lambda>\lambda$ and $\epsilon<0$.

	Combining \eqref{ineq:cond_mu_1}, \eqref{ineq:cond_mu_2} and \eqref{ineq:cond_Delta}, we have 
	\[ \sum_{i=1}^N \| x_{i,k+1} - x^*\|^2\leq N e_0^2\gamma^{2k+2},\]
	if $0<\mu_0\leq \min\{\frac{e_0}{2\beta - \rho e_0},\frac{q}{10\sqrt{N}(a\lambda +L^2+\frac{ ac\Lambda}{1-\Lambda})}\}$,  $ \gamma =\lambda^{\delta}$ and $ \|\Delta_0\| < \frac{\frac{2 }{\Gamma}\beta e_0 - \rho e_0^2}{2(L+\beta) c}\lambda$.

	Lastly, we need to verify \eqref{ineq:linear_reg2} for $k+1$.  Since $\|\bar{x}_{k+1} - x^*\|^2\leq \frac{1}{N}\sum_{i=1}^N \| x_{i,k+1} - x^*\|^2\leq e_0^2\gamma^{2k+2},$ it follows from \eqref{ineq:estimation_of_Delta_2} that
	{\small 
	\begin{align*}
	&\quad\| x_{i,k+1} - x^* \|\\
	&\leq \| x_{i,k+1} - \bar{x}_{k+1}\| + \|\bar{x}_{k+1} - x^*\|\\
	&\leq \|\Delta_{k+1}\| + e_0 \gamma^{k+1}\\
	&\leq \left(  \frac{c}{\lambda}  \|\Delta_0\| +     \frac{\sqrt{N}L}{\lambda^2}( c \frac{\gamma^{\frac{1}{\delta}-1}}{1- \gamma^{\frac{1}{\delta}-1}} + \lambda)\mu_0  \right)\gamma^{k+1} + e_0 \gamma^{k+1}
	\end{align*}
	}
	Using \eqref{ineq_induction_step2}, one has \[( c \frac{\gamma^{\frac{1}{\delta}-1}}{1- \gamma^{\frac{1}{\delta}-1}} + \lambda)\mu_0< \frac{q}{10\sqrt{N} a}\leq \frac{\beta e_0}{5\sqrt{N} a}.   \]
	Therefore, we have
	\[ \| x_{i,k+1} - x^* \|\leq\left(  \frac{c}{\lambda}  \|\Delta_0\| +   \frac{\beta}{10(L+\beta)}e_0 \right)\gamma^{k+1} + e_0 \gamma^{k+1}. \]
	Since $ \|\Delta_0\| < \frac{\frac{2 }{\Gamma} \beta e_0 - \rho e_0^2}{2(L+\beta) c}\lambda $, it follows that \[ \bad
	&\quad\| x_{i,k+1} - x^* \|\\
	&\leq\left(  \frac{\frac{2 }{\Gamma}\beta e_0 - \rho e_0^2}{2(L+\beta) }+   \frac{\beta}{10(L+\beta)}e_0 \right)\gamma^{k+1} + e_0 \gamma^{k+1}\\
	&\leq (\frac{1}{2\Gamma}+\frac{21}{20})e_0 \gamma^{k+1}\\
	&\leq \Gamma e_0 \gamma^{k+1},
	\ead\]
	where the second inequality follows from  $\beta\leq L$ and the last inequality holds since $\Gamma\geq \sqrt{2}$. 
\end{proof}

\begin{proof}[\textbf{Proof of \cref{thm:convergence_stoDPSM}}]
The proof is quite similar to that of algorithm \eqref{alg:dis_subgrad_orgin}. We explain the main steps below. First, we have the same consensus lemma as Lemma  \ref{lemma:consensus1}.  
Substituting $z = \hat{v}_{i,k}$ into \eqref{eq_derivation1} and taking expectation conditioned on $k$, we obtain
\begin{equation}\label{eq_derivation3-sto}
\bad
&\quad \E	\varphi_{t}\left(x_{i ,k+1}\right) \\
&\leq  f(\hat{v}_{i,k})+\E \frac{1}{2 t}\left\|x_{i , k+1}-\hat{v}_{i,k}\right\|^{2} \\
&\leq  f(\hat{v}_{i,k}) + \frac{1}{2 t} \E\left(   \normtwo{ v_{i,k} - \hat{v}_{i,k}  }^2 - 2\alpha_k\inp{v_{i,k} - \hat{v}_{i,k}  }{\xi_{i,k}} \right.\\
&\quad\left. + \alpha_k^2\normtwo{\xi_{i,k}}^2   \right)\\
&\leq \varphi_{t}(v_{i,k}) - \frac{\alpha_k}{t}\inp{v_{i,k} - \hat{v}_{i,k}  }{g_{i,k}} + \frac{\alpha_k^2}{t}L^2.
\ead
\end{equation}
Then, the remaining parts of the proof are the same as  that of Theorem \ref{thm:convergence_1}.
\end{proof}

\bibliographystyle{ieeetr}
\bibliography{all}

\end{document}